\documentclass[a4paper, 10pt, twoside, notitlepage]{amsart}

\usepackage[utf8]{inputenc}
\usepackage{color}
\usepackage{amsmath} 
\usepackage{amssymb} 
\usepackage{amsthm}
\usepackage{geometry}
\usepackage{graphicx}
\graphicspath{{figures/}}
\usepackage{esint}
\usepackage[colorlinks=true,linkcolor=blue]{hyperref}

\theoremstyle{plain}
\newtheorem{thm}{Theorem}
\newtheorem{prop}{Proposition}[section]
\newtheorem{lem}[prop]{Lemma}

\newtheorem{defi}[prop]{Definition}
\newtheorem{rmk}[prop]{Remark}

\newcommand {\R} {\mathbb{R}} 
 \newcommand {\N} {\mathbb{N}}
\newcommand {\C} {\mathbb{C}} 
\newcommand {\Sp} {\mathbb{S}}
\newcommand {\p} {\partial}

\newcommand {\supp} {\text{supp}}
\newcommand {\diam} {\text{diam}}

\newcommand {\conv} {\text{conv}}

\newcommand {\uaux}{u_{\rm aux}}
\newcommand {\chiaux}{\chi_{\rm aux}}

\definecolor{BudGreen}{RGB}{112, 174, 110}
\definecolor{orange}{rgb}{1, 0.5, 0}

\DeclareMathOperator{\Per}{Per}

\DeclareMathOperator {\dist} {dist}

\DeclareMathOperator{\F} {\mathcal{F}}
\DeclareMathOperator{\B} {\mathcal{B}}

\title[Scaling of the Cubic-to-Tetragonal Phase Transformation with Dirichlet Data]{On the Scaling of the Cubic-to-Tetragonal Phase Transformation with Displacement Boundary Conditions}

\author{Angkana Rüland}
\address{Institute for Applied Mathematics and Hausdorff Center for Mathematics, University of Bonn, Endenicher Allee 60, 53115 Bonn, Germany}
\email{rueland@uni-bonn.de}

\author{Antonio Tribuzio}
\address{Institute for Applied Mathematics, University of Bonn, Endenicher Allee 60, 53115 Bonn, Germany}
\email{tribuzio@iam.uni-bonn.de}

\date{}

\begin{document}

\begin{abstract}
We provide (upper and lower) scaling bounds for a singular perturbation model for the cubic-to-tetragonal phase transformation with (partial) displacement boundary data. We illustrate that the \emph{order of lamination} of the affine displacement data determines the complexity of the microstructure. As in \cite{RT21} we heavily exploit careful Fourier space localization methods in distinguishing between the different lamination orders in the data.
\end{abstract}

\maketitle
\tableofcontents
\addtocontents{toc}{\setcounter{tocdepth}{1}}

\section{Introduction}

Motivated by complex microstructures in shape-memory alloys, in this article we study the scaling behaviour of a variational problem of singular perturbation type modelling the cubic-to-tetragonal phase transformation in the low temperature regime with prescribed displacement boundary conditions. The cubic-to-tetragonal phase transformation is a frequently occurring phase transformation in metal alloys such as NiMn and has been investigated in the mathematical literature as a prototypical, vector-valued non-quasiconvex variational problem, displaying complex microstructure \cite{B, Ball:ESOMAT, M1}. Here both qualitative and quantitative rigidity estimates have been deduced \cite{DM2, K, CDK, CO09, CO12, KKO13, KO19,TS21a, TS21b}. Quantitative results have been obtained in settings with periodic \cite{CO09} and nucleation data \cite{KKO13, KO19} as well as in the form of ``interior rigidity estimates'' \cite{CO12}. 

It is the purpose of this note to investigate the setting with \emph{displacement data}, i.e., Dirichlet data, and to illustrate that some of the main findings from \cite{RT21} (which build on the earlier works from \cite{CO09, CO12, KKO13, RT21,KW16} and systematize these in settings without gauge invariances) are also true in settings \emph{with non-trivial gauge groups}, i.e., in particular in settings with $Skew(3)$ invariance, which encodes infinitesimal frame indifference. More precisely, we show that also in the setting of the geometrically linearized cubic-to-tetragonal phase transformation with $Skew(3)$ invariance the relevant quantity determining the scaling behaviour and the complexity of the microstructure in the singular perturbation problem with displacement data is given by the \emph{order of lamination} of the Dirichlet data.

\subsection{The model}
Let us describe our mathematical model for the cubic-to-tetragonal phase transformation in more detail. In the exactly stress-free setting the cubic-to-tetragonal phase transformation corresponds to studying the structure of displacements satisfying the differential inclusion
\begin{equation}\label{eq:diff-incl}
e(\nabla u) \in \left\{\begin{pmatrix}-2&0&0\\0&1&0\\0&0&1\end{pmatrix},\begin{pmatrix}1&0&0\\0&-2&0\\0&0&1\end{pmatrix},\begin{pmatrix}1&0&0\\0&1&0\\0&0&-2\end{pmatrix}\right\}=:K \mbox{ a.e. in } \Omega.
\end{equation}
Here, $\Omega \subset \R^3$ is an open, Lipschitz domain. The matrix $e(M)=\frac{1}{2}(M+M^t)$ denotes the symmetric part of $M\in\R^{3\times 3}$ and the \emph{symmetrized gradient} $e(\nabla u)$ of the \emph{displacement} $u:\Omega \rightarrow \R^3$ physically corresponds to the (infinitesimal) \emph{strain}. The three matrices 
\begin{align*}
e^{(1)}:= \begin{pmatrix}-2&0&0\\0&1&0\\0&0&1\end{pmatrix}, \ e^{(2)}:= \begin{pmatrix}1&0&0\\0&-2&0\\0&0&1\end{pmatrix}, \ e^{(3)}:=\begin{pmatrix}1&0&0\\0&1&0\\0&0&-2\end{pmatrix},
\end{align*}
 model the \emph{variants of martensite}, i.e., the exactly stress-free states of the material in the low-temperature regime.

For later convenience, we note that this differential inclusion can also be expressed in the stress-free setting with the help of characteristic functions $\chi_j\in L^{\infty}(\Omega;\{0,1\})$, $j\in\{1,2,3\}$, with $\chi_1+\chi_2+\chi_3=1$, as follows
\begin{equation}
\label{eq:stress_free2}
\nabla u = \chi:=
\begin{pmatrix}
-2 \chi_1 + \chi_2 + \chi_3 & 0 & 0\\
0& \chi_1 - 2\chi_2 + \chi_3 & 0 \\
0 & 0 & \chi_1 + \chi_2 -2 \chi_3
\end{pmatrix}  \mbox{ a.e. in } \Omega.
\end{equation}
Here the functions $\chi_j $, $j\in\{1,2,3\}$, will be referred to as the \emph{phase indicators} corresponding to the respective variants of martensite. The structure of the set $K$ is well-known \cite{B}: The matrices $e^{(j)}$ are pairwise symmetrized rank-one connected, i.e., adopting the notation from \cite{KKO13}, for each $i, j \in \{1,2,3\}$, $i\neq j$, there exist vectors $b_{ij}, b_{ji} \in \mathbb{S}^2$ such that
\begin{align*}
e^{(i)}- e^{(j)} = 3\epsilon_{ijk}\left( b_{ij} \otimes b_{ji} + b_{ji} \otimes b_{ij} \right).
\end{align*}
Here $\epsilon_{ijk}$ denotes the antisymmetric $\epsilon$-tensor, i.e., 
\begin{align*}
\epsilon_{ijk}:=
\left\{ 
\begin{array}{ll}
1 &\mbox{ if } (i,j,k) \mbox{ is an even permutation of $(1,2,3)$},\\
-1 &\mbox{ if } (i,j,k) \mbox{ is an odd permutation of $(1,2,3)$},\\
0 &\mbox{ else}.
\end{array}
\right.
\end{align*}

Hence, in particular, so-called \emph{simple laminates} or \emph{twins} arise as solutions to \eqref{eq:diff-incl}. These correspond to one-dimensional displacements $u:\Omega \rightarrow \R^3$ of the form
\begin{align*}
u(x):= Ax+f(v \cdot x), \ v \in \{b_{ij},b_{ji}\}, \ i,j\in\{1,2,3\}, \ i\neq j,
\end{align*}
with some $f: \R \rightarrow \R^3$ and such that the infinitesimal strain alternates between the two values $e^{(i)}, e^{(j)}$.
The vectors $b_{ij}$ are given by
\begin{align}
\label{eq:normals}
\begin{split}
&b_{12} = \frac{1}{\sqrt{2}}\begin{pmatrix}
1 \\ 1 \\ 0 \end{pmatrix},\ b_{21} = \frac{1}{\sqrt{2}}\begin{pmatrix}\
-1 \\ 1 \\0 \end{pmatrix}, \ b_{31} = \frac{1}{\sqrt{2}}\begin{pmatrix}\
1 \\ 0 \\1 \end{pmatrix},\\
&b_{13} = \frac{1}{\sqrt{2}}\begin{pmatrix}
1 \\ 0 \\-1 \end{pmatrix}, \ 
b_{23} = \frac{1}{\sqrt{2}}\begin{pmatrix}
0 \\ 1 \\1 \end{pmatrix},\
b_{32} = \frac{1}{\sqrt{2}}\begin{pmatrix}
0 \\ -1 \\ 1 \end{pmatrix}.
\end{split}
\end{align}
Borrowing the notation from \cite{KKO13}, for $i,j,k\in\{1,2,3\}$, $i\neq j$, $j\neq k$ and $i\neq k$, we collect these into the sets 
\begin{align}
\label{eq:Bij}
\mathcal{B}_{ij}:= \{b_{ij}, b_{ji}\} \mbox{ and } \mathcal{B}_i:= \mathcal{B}_{ij}\cup \mathcal{B}_{ik},
\end{align}
consisting of all lamination normals between the strains $e^{(i)}, e^{(j)}$ and all lamination normals involving the martensitic variant $e^{(i)}$, respectively.

Motivated by the formulation \eqref{eq:stress_free2} and using the phase indicator energy models from \cite[Chapter 11]{B} and \cite{CO09, CO12}, we study singularly perturbed energies of the form
\begin{align}\label{eq:total-energy}
\begin{split}
E_{\epsilon}(\chi) &:= \inf_{u\in H^1_0(\Omega;\R^3)} (E_{el}(u,\chi) + \epsilon E_{surf}(\chi)), \\ 
\mbox{with }
E_{el}(u,\chi&):=\int_{\Omega} \big|e(\nabla u)-\chi\big|^2 dx, \ E_{surf}(\chi):=  |D\chi|(\Omega),
\end{split}
\end{align}
with $\chi$ as in \eqref{eq:stress_free2} and with $|D\chi|(\Omega)$ denoting the total variation norm of $D \chi$ in $\Omega$. The choice of the data space $u\in H^1_0(\Omega;\R^3)$ here physically corresponds to the prescription of fixed displacement data given by the \emph{austenite phase}, i.e., the high temperature phase. 
As our main objective, we seek to deduce a scaling law for $E_{\epsilon}(\chi) $ in the small parameter $\epsilon>0$ with these prescribed austenite displacement data and, more generally, with prescribed affine lamination boundary conditions (in the second order lamination convex hull).
We split the discussion of these bounds into two parts: In the next section, we first discuss and present lower bounds whose scaling in $\epsilon$ is determined by the lamination order of the displacement data. In Section \ref{sec:upperintro} we then turn to the upper bounds for which we consider slightly different boundary data (partially displacement, partially periodic) for technical reasons.

\subsection{Lower bounds}
\label{sec:lowerbounds}

We begin by discussing the lower bounds in the setting with austenite displacement boundary conditions. In this setting our main result reads as follows.

\begin{thm}
\label{thm:main_zero}
Let $\Omega \subset \R^3$ be an open, bounded Lipschitz domain.
Let $E_{\epsilon}(\chi)$ be as in \eqref{eq:total-energy}. Then, there exist $\epsilon_0=\epsilon_0(\Omega,K)>0$ and $C=C(\Omega,K)>1$ such that for any $\epsilon \in (0,\epsilon_0)$ it holds
\begin{align*}
C^{-1} \epsilon^{\frac{1}{2}} \leq \inf\limits_{\chi \in BV(\Omega; K)} E_{\epsilon}(\chi).
\end{align*}
\end{thm}

Let us comment on this scaling result: As in the results in \cite{RT21}, Theorem \ref{thm:main_zero} quantifies the lower bound scaling behaviour in the singular perturbation parameter $\epsilon>0$ in terms of the \emph{order of lamination} of the displacement data. Indeed, as in \cite{KKO13} we note that the austenite phase (corresponding to the zero matrix) is a \emph{second} order laminate (c.f. Definition \ref{defi:lam} in Section \ref{sec:not}) and can be expressed as
\begin{align*}
\textbf{0} = \frac{1}{2}\left(\frac{1}{3}e^{(i)} + \frac{2}{3}e^{(j)} \right) + \frac{1}{2}\left(\frac{1}{3}e^{(i)} + \frac{2}{3}e^{(k)} \right), \ i,j,k \in \{1,2,3\}, \ i\neq j, \ j\neq k, \ i \neq k.
\end{align*}
While a scaling of the order $\epsilon^{\frac{2}{3}}$ (as in \cite{CO09, CO12} and as in the seminal results \cite{KM1, KM2}) corresponds to \emph{simple (i.e., first order) laminate boundary conditions}, the scaling of the order $\epsilon^{\frac{1}{2}}$ indeed encodes the presence of laminates of \emph{second order}. This directly matches the results from \cite[Theorem 3]{RT21} in which $\epsilon^{\frac{1}{2}}$ matching upper and lower scaling results were proved for laminates of second order in settings without gauge invariances.
As in \cite{RT21, RT22, RT23} we deduce the result of Theorem \ref{thm:main_zero} by systematically localizing the possible domains of concentration of the Fourier support of the components of the phase indicator $\chi$. In particular, this shows that the arguments from \cite{RT21} in combination with the use of the central trilinear quantity which had been identified in \cite{KKO13} are sufficiently robust to be applied in settings \emph{with} geometrically linearized frame indifference, i.e., in the presence of the gauge group $Skew(3)$.

In order to substantiate the claim that, indeed, the order of lamination is the only determining factor in the scaling law for the singularly perturbed energies and that the scaling law of order $\epsilon^{\frac{1}{2}}$ is no co-incidence, we complement the physically most interesting setting of austenite displacement boundary conditions from Theorem \ref{thm:main_zero} by general affine boundary data in the first and second order lamination convex hulls:

\begin{thm}
\label{thm:main_general}
Let $\Omega \subset \R^3$ be an open, bounded Lipschitz domain.
Let $E_{el}(u,\chi), E_{surf}(\chi)$ be as in \eqref{eq:total-energy}. 
\begin{itemize}
\item[(i)] Let $F\in K^{(2)}\setminus K^{(1)}=\text{intconv}\{e^{(1)}, e^{(2)}, e^{(3)}\}$ (see Definition \ref{defi:lam} in Section \ref{sec:not}). Let 
\begin{equation}\label{eq:BCset}
\mathcal{A}_F:=\{u\in H^{1}_{loc}(\R^3; \R^3): \ u(x) = F x + b \mbox{ in } \R^3 \setminus \overline{\Omega} \mbox{ for some } b\in \R^3\}.
\end{equation}
Then, there exist $\epsilon_0>0$ (depending only on $\Omega$ and $\dist(F, K^{(1)})$) and $C=C(\Omega,F,K)>1$ such that for any $\epsilon \in (0,\epsilon_0)$ it holds
\begin{align*}
C^{-1} \epsilon^{\frac{1}{2}} \leq \inf\limits_{\chi \in BV(\Omega; K)} \inf\limits_{u \in \mathcal{A}_F} (E_{el}(u,\chi) + \epsilon E_{surf}(\chi)) .
\end{align*}
\item[(ii)]  Let 
\begin{align*}
F\in K^{(1)} \setminus K
&=(e^{(1)},e^{(2)})\cup (e^{(2)}, e^{(3)})\cup (e^{(1)}, e^{(3)})\\
&= \text{conv}\{e^{(1)}, e^{(2)}, e^{(3)}\} \setminus \big(\text{intconv}\{e^{(1)}, e^{(2)}, e^{(3)}\}\cup\{e^{(1)},e^{(2)},e^{(3)}\}\big)
\end{align*}
(see Definition \ref{defi:lam} in Section \ref{sec:not}).
Then, there exist $\epsilon_0>0$ (depending only on $\Omega$ and on $\dist(F,K)$) and $C=C(\Omega,F,K)>1$ such that for any $\epsilon \in (0,\epsilon_0)$ it holds
\begin{align*}
C^{-1} \epsilon^{\frac{2}{3}} \leq  \inf\limits_{\chi \in BV(\Omega; K)} \inf\limits_{u \in \mathcal{A}_F} (E_{el}(u,\chi) + \epsilon E_{surf}(\chi)).
\end{align*}
\end{itemize}
\end{thm}

Let us comment on this result: We first note that part (i) highlights that the scaling found in Theorem \ref{thm:main_zero} holds for \emph{any} boundary data in the \emph{second} order lamination convex hull. Moreover, all boundary data in the \emph{first} order lamination convex hull result in an $\epsilon^{\frac{2}{3}}$ lower scaling bound as proved in part (ii).

As in \cite{RT21} the theorem thus gives strong indication that in the setting of affine displacement boundary data the scaling bounds and thus the complexity of microstructures for the cubic-to-tetragonal phase transformation are purely determined by the complexity of the displacement data encoded in their \emph{lamination order}. In the next section, we will substantiate that one should expect all of these bounds to be sharp by providing matching upper bound constructions (for slightly modified boundary data).

As in \cite{RT21, RT22, RT23}, the arguments for Theorems \ref{thm:main_zero}, \ref{thm:main_general} systematically use the reduction of the Fourier concentration domains by relying on a two-step bootstrap procedure for double and a single step Fourier localization argument for simple laminate data.

\subsection{Upper bounds}
\label{sec:upperintro}

In order to give evidence of the optimality of the deduced lower scaling bounds, we also discuss an upper bound construction. Due to technical difficulties, for boundary data which are laminates of second order we here focus on settings with only partial displacement data and -- since this is the physically most relevant setting -- only consider the case of zero boundary data (any boundary condition in the interior of the lamination convex hull could be treated analogously). 
Moreover, since due to the three-dimensionality of the problem at hand, the optimal upper bound constructions can be rather complex (c.f. the proof of Proposition 6.3 in \cite{RT21}), we further focus on a specific choice of domain. We refer to the beginning of Section \ref{sec:upper} for further comments on this and to Remark \ref{rmk:subopt} for observations on deducing ``direct'' but sub-optimal bounds in the full Dirichlet data setting.

Let us outline the precise set-up: Recalling that
\begin{align*} 
\textbf{0} = \frac{1}{2}\left(\frac{1}{3}e^{(1)} + \frac{2}{3}e^{(2)} \right) + \frac{1}{2}\left(\frac{1}{3}e^{(1)} + \frac{2}{3}e^{(3)} \right), 
\end{align*}
we seek to work with a second order lamination construction, laminating $e^{(A)}:= \frac{1}{3}e^{(1)} + \frac{2}{3}e^{(2)}$ and $e^{(B)}:=\frac{1}{3}e^{(1)} + \frac{2}{3}e^{(3)} $ in an outer (branched) laminate oriented in the direction $b_{32}$ first, and then filling these with a second order laminate between $e^{(1)}, e^{(2)}$ and $e^{(1)}, e^{(3)}$ in directions $b_{21}$ and $b_{13}$, respectively. Compared to the setting in \cite{RT21} these directions are not all orthogonal but only transversal. Moreover, the second lamination is necessary in \emph{both} of the outer, first order laminates.
In order to avoid additional technicalities originating from this, we fix the domain in which we carry out this construction as follows:
We define the orthonormal basis $\mathcal{R}:=\{n,b_{32},d\}$ where $n,d\in\mathbb{S}^2$ are given by
\begin{equation}
\label{eq:basis}
n:=\frac{1}{\sqrt{6}}\begin{pmatrix}-2\\1\\1\end{pmatrix},
\quad
d:=\frac{1}{\sqrt{3}}\begin{pmatrix}1\\1\\1\end{pmatrix}.
\end{equation}
We highlight that $d$ is orthogonal to all the chosen directions of lamination, i.e., $b_{21},b_{13},b_{32}\in \text{span}(d)^\perp$.
With this observation, we then define
\begin{align}
\label{eq:domain}
\Omega:=\Big(-\frac{1}{2},\frac{1}{2}\Big)n\times\Big(-\frac{1}{2},\frac{1}{2}\Big)b_{32}\times\Big(-\frac{1}{2},\frac{1}{2}\Big)d.
\end{align}

Further, working with partial Dirichlet and partial periodic boundary conditions, we introduce the following notation: For any $\Omega'\subseteq\Omega$ we write
\begin{align}
\label{eq:part_boundary}
\p_0\Omega' := \p\Omega' \setminus \Big\{x\cdot d=\pm\frac{1}{2}\Big\}.
\end{align}

With this notation in hand, we formulate our main result on the scaling of an upper bound construction.

\begin{thm}[Upper bound construction, second order laminates]
\label{thm:upper}
Let $\Omega \subset \R^3$ be as in \eqref{eq:domain}.
Let $E_{el}(u,\chi), E_{surf}(\chi)$ be as in \eqref{eq:total-energy}.  Then there exist $\epsilon_0=\epsilon_0(\Omega,K)>0$ and $C=C(\Omega,K)>1$ such that for every $\epsilon\in(0,\epsilon_0)$ there exist mappings $u_\epsilon\in C^{0,1}(\R^3; \R^{3})$, $\chi_{\epsilon} \in BV_{loc}(\R^3; K)$ with 
\begin{align*}
E_{el}(u_\epsilon,\chi_\epsilon) + \epsilon E_{surf}(\chi_\epsilon) \leq C \epsilon^{\frac{1}{2}}
\end{align*}
and such that $u_\epsilon, \chi_\epsilon$ are one-periodic in the $d$ direction and satisfy $u_\epsilon(x) = 0$ for every $x \in \p_0 \Omega$. 
\end{thm}

\begin{rmk}
\label{rmk:boundary}
We remark that our boundary conditions will in fact be slightly stronger: It will hold that $u_{\epsilon}, \chi_{\epsilon}$ are not only periodic, but actually constant in the $d$ direction.
\end{rmk}

Moreover, for completeness, we complement Theorem \ref{thm:upper} with an upper bound construction for a first order laminate. For simplicity, we here only consider a particular boundary datum from $K^{(1)}$; the scaling for general data in $K^{(1)}$ can be obtained similarly.

\begin{thm}[Upper bound construction, first order laminates]
\label{thm:upper2}
Let $\Omega \subset \R^3$ be as in \eqref{eq:domain} and let $F= \frac{1}{2} e^{(2)} + \frac{1}{2} e^{(3)}$.
Let $E_{el}(u,\chi), E_{surf}(\chi)$ be as in \eqref{eq:total-energy}.  Then there exist $\epsilon_0=\epsilon_0(\Omega,F,K)>0$ and $C=C(\Omega,F,K)>1$ such that for every $\epsilon\in(0,\epsilon_0)$ there exist mappings $u_\epsilon\in C^{0,1}(\Omega; \R^{3})\cap\mathcal{A}_F$, $\chi_{\epsilon} \in BV(\Omega; \{e^{(2)},e^{(3)}\})$ with 
\begin{align*}
E_{el}(u_\epsilon,\chi_\epsilon) + \epsilon E_{surf}(\chi_\epsilon) \leq C \epsilon^{\frac{2}{3}}. 
\end{align*}
\end{thm}

The upper bound constructions from Theorems \ref{thm:upper} and \ref{thm:upper2} are based on the ones from \cite{RT21} (and the earlier versions from \cite{CC15,CO09, CO12}).

Let us comment on the role of these upper bound constructions:
Firstly, we emphasize that Theorems \ref{thm:upper}, \ref{thm:upper2} (together with Theorems \ref{thm:main_zero}, \ref{thm:main_general}) indeed give strong credence to the role of the order of lamination of the prescribed displacement Dirichlet data as the determining factor of the scaling law for the singularly perturbed energy from above. On the one hand, for the simple laminate regime for planar configurations the $\epsilon^{\frac{2}{3}}$ scaling is a well-known result in the literature (c.f. \cite{CO09,CO12}). In what follows, in the proof of Theorem \ref{thm:upper2} we modify this slightly by using the proof of Proposition 6.3 from \cite{RT21} to include the full Dirichlet conditions. 
On the other hand, our result from Theorem \ref{thm:upper} provides an $\epsilon^{\frac{1}{2}}$ scaling for the second order laminate setting (c.f. also \cite{KW16}).
Hence, in both regimes -- the one for simple laminate data and the one for data in the second order lamination convex hull -- the scaling behaviour exactly matches the behaviour for first and second order laminates which had been deduced in \cite{RT21} for singular perturbation models without gauge invariances. 
Hence, in spite of the fact that Theorem \ref{thm:upper} does not feature complete displacement data, the given construction does provide an $\epsilon^\frac{1}{2}$ scaling for a second order lamination.
We thus view Theorem \ref{thm:upper} as a strong indication that -- apart from technical difficulties -- the lower scaling exponent from Theorem \ref{thm:main_zero} is indeed optimal. 
As a consequence, this thus also underpins the role of simplified models as in \cite{RT21,RT22, RT23} in the study of more realistic non-quasiconvex multiwell energies.
Further, we highlight that to the best of the authors' knowledge, the construction from Theorem \ref{thm:upper} provides a first result of a scaling bound for \emph{(branched) second order laminates} in \emph{fixed domains} in models of linearized elasticity (\cite{KKO13} provides a nucleation-type second order laminate construction in which essentially first order laminates are prevalent, \cite{KW16} provides a similar second order construction as ours in the setting of compliance minimization).
Finally, also from an experimental point of view, due to the planarity of the interactions involved in second order laminate constructions in the cubic-to-tetragonal phase transformation, the mixed Dirichlet-periodic boundary conditions could give the right intuition for interesting experimental settings such as for situations involving laminates which only refine towards one habit plane (instead of investigating a nucleus which is completely immersed in a sea of austenite).

We further note that it is expected that this behaviour persists for other variants of singularly perturbed energies (e.g. singular perturbation models involving diffuse surface energies) whose exploration, for clarity of exposition, we however postpone to future work.

\subsection{Relation to the results in the literature}

As a prototypical non-quasiconvex multiwell problem in materials science and the calculus of variations, the cubic-to-tetragonal phase transformation has been investigated rather intensively: The foundational result \cite{DM2} provided the first qualitative rigidity result, showing that in the geometrically linearized theory the only stress-free solutions to the differential inclusion \eqref{eq:diff-incl} are (locally) given by \emph{simple laminates} or \emph{twins} with the normals from \eqref{eq:normals}. In this sense the differential inclusion \eqref{eq:diff-incl} is (rather) rigid in the \emph{geometrically linearized} theory. In the \emph{geometrically nonlinear theory of elasticity} a rather striking difference emerges: On the one hand, if for the associated differential inclusion the $Skew(3)$ invariance is replaced by a $SO(3)$ invariance and if no additional regularity condition is assumed for the deformation field, then the differential inclusion problem becomes extremely \emph{flexible} in that it permits substantially more complicated solutions than simple laminates \cite{CDK}. These are obtained by a convex integration scheme. If, on the other hand, in the nonlinear differential inclusion the deformation field is such that the associated deformation gradient is $BV$ regular, then, as shown in \cite{K}, the rigidity from \cite{DM2} is recovered in that again only local simple laminates emerge as possible solutions to the associated differential inclusion. As a consequence, while the geometrically linearized cubic-to-tetragonal phase transformation is \emph{rigid} without any additional regularity requirement, the geometrically nonlinear version displays a \emph{dichotomy between rigidity and flexibility} (c.f also \cite{DM1,R16,RTZ19, RZZ16,DPR20} for related results in between rigidity and flexibility).

Building on the seminal articles \cite{KM1, KM2} also the \emph{quantitative analysis} of the scaling behaviour of non-quasiconvex, singularly perturbed energies of the type \eqref{eq:total-energy} has attracted substantial activity as prototypical multi-well problems arising from materials science. Here, particularly the geometrically linearized cubic-to-tetragonal phase transformation has been studied with various objectives: The fundamental articles \cite{CO09, CO12} capture and quantify the rigidity results from \cite{DM2} in the periodic setting and in the form of ``localized interior'' rigidity estimates; the articles \cite{KKO13, KO19} study its nucleation behaviour in the form of a precise scaling law by establishing (anisotropic) isoperimetric type estimates and in \cite{BG15} nucleation behaviour in corner domains is considered.
Moreover, the articles \cite{TS21a, TS21b} explore the fine properties of this phase transformation in a regime displaying the $\epsilon^{\frac{2}{3}}$ scaling behaviour of the simple laminate scaling.

In this context, the results from \cite{CO09, CO12} and \cite{KKO13} are closely related to our setting. However, in contrast to our result, on the one hand the main results from \cite{CO09, CO12} seek to quantify the rigidity estimates from \cite{DM2} and thus focus on \emph{first order laminates}. On the other hand, the result from \cite{KKO13} focuses on the \emph{nucleation behaviour} for martensite in a matrix of austenite. It thus deals with a \emph{nucleation-type boundary condition} for a second order laminate. While it is thus quite close to our result, as an isoperimetric problem, it however contains a further degree of freedom in the choice of the nucleation domain (a lens-shaped domain satisfies this for instance) and thus displays a \emph{different (lower) bound} in the (non-dimensionalized) surface energy parameter. Contrary to the setting in \cite{KKO13}, by prescribing displacement (i.e., Dirichlet) conditions, microstructure is enforced \emph{more strongly} in our setting and thus becomes \emph{more expensive} in the Dirichlet case. A similar phenomenon had been deduced for simplified models without gauge invariances in \cite{RT21, RT23} in which even higher order laminates were treated systematically and in a unified framework. 
The present article shows that also for the geometrically linearized cubic-to-tetragonal phase transformation -- which now has \emph{gauge invariances} -- the \emph{complexity of the Dirichlet boundary conditions} (in the form of their lamination order) is the key factor in the determination of the complexity of the scaling law in our fixed domain setting. We view this observation as one of the main contributions of this article.

Let us emphasize that there are many further contributions on the study of the dichotomy between rigidity and flexibility and the exploration of (scaling laws for) complex microstructures related to shape-memory alloys. As a non-exhaustive list we refer to the articles \cite{ContiSchweizer06, ContiSchweizer06a, DF20,KLLR19, CDZ17,C,Pompe, CKZ17, KK,CDPRZZ20,CT05,K,CM99,L06, L01, BJ92, Ball:ESOMAT,B2,KMS03,CS13,CS15,BK16,CC15,CDMZ20,
R16,RS23,Rue16b,RTZ19,RT22,AKKR22,W97,C99,R22,RRT23} and the references therein.
Moreover, we highlight that similar ideas and techniques have also been used in important related models such as in compliance minimization \cite{KW14,KW16, PW21} or, for instance, in models for micromagnetics \cite{CKO99,OV10,KN18,KS22,GZ23}.

\subsection{Outline of the article}
The remainder of the article is structured as follows: In Section \ref{sec:not} we briefly recall the notation for the cubic-to-tetragonal phase transformation and some preliminary results. Next, in Section \ref{sec:lower} we provide the lower bound estimates of Theorems \ref{thm:main_zero}, \ref{thm:main_general} by exploiting systematic Fourier bounds. In Section \ref{sec:upper} we finally present the upper bound constructions of Theorems \ref{thm:upper}, \ref{thm:upper2}.

\section{Notation and Preliminaries}
\label{sec:not}

In this section we recall some basic notation and properties of the cubic-to-tetragonal phase transformation.

\begin{defi}[Laminates, order of lamination]
\label{defi:lam}
Let $K\subset \R^{n\times n}_{sym}$ be a compact set. We set $K^{(0)}:= K$ and for $j\geq 1$ 
\begin{align*}
K^{(j)}
&:=\{e \in \R^{n\times n}_{sym}: \  \mbox{ there exist } e_1, e_2 \in K^{(j-1)}, \ \lambda \in [0,1], \ a, b \in \R^{n}\setminus \{0\} \\
& \quad \mbox{ such that } e = \lambda e_1 + (1-\lambda) e_2, \mbox{ and } e_1 - e_2 = \frac{1}{2}(a\otimes b + b \otimes a)\}.
\end{align*}
The \emph{(symmetrized) lamination convex hull} $K^{(lc,sym)}$ is then defined as 
\begin{align*}
K^{(lc,sym)}:= \bigcup\limits_{j=0}^{\infty} K^{(j)}.
\end{align*}
For $j\geq 1$ we denote the elements of $K^{(j)}\setminus K^{(j-1)}$ as \emph{laminates of order $j$}. If $j=1$ we also refer to the elements in $K^{(1)}\setminus K$ as \emph{simple laminates}.
\end{defi}

It is well-known (c.f. \cite{B}) that for the geometrically linearized cubic-to-tetragonal phase transformation and thus for $K:=\{e^{(1)}, e^{(2)}, e^{(3)}\}$ as in \eqref{eq:diff-incl} it holds that $K^{(lc,sym)} = K^{(2)}$ and that 
\begin{align*}
K^{(1)} = \text{conv}(\{e^{(1)},e^{(2)}, e^{(3)}\})\setminus \text{intconv}(\{e^{(1)},e^{(2)}, e^{(3)}\}), \ K^{(2)}\setminus K^{(1)} = \text{intconv}(\{e^{(1)},e^{(2)}, e^{(3)}\}).
\end{align*}
In particular, $K^{(1)}$ is given by the line segments connecting the wells. For notational simplicity we also denote this as
\begin{align*}
\text{conv}(\{e^{(1)},e^{(2)}, e^{(3)}\})\setminus \text{intconv}(\{e^{(1)},e^{(2)}, e^{(3)}\})=[e^{(1)},e^{(2)}]\cup [e^{(2)}, e^{(3)}]\cup [e^{(1)}, e^{(3)}].
\end{align*}

\section{The Lower Bound: Fourier Estimates}
\label{sec:lower}

In this section we present the proofs of Theorems \ref{thm:main_zero} and \ref{thm:main_general}. To this end, we first recall a characterization of the elastic energy which had already been used in \cite{CO09} and \cite{KKO13}. We then use coercivity and high frequency bounds for a first frequency truncation similarly as in \cite{RT21, RT23}. In order to deduce the bounds for second order laminate boundary data this is then combined with a second frequency truncation. Here a major role is played by a trilinear quantity (in the phase indicators) which had been introduced in \cite{KKO13} and which encodes information on second order laminates. We derive sufficiently strong frequency localized control for it.

\subsection{The elastic energy}

Let $u\in L^2(\R^3)$. We define its Fourier transform (via the standard density argument) using the following convention
$$
\F u(k) := (2\pi)^{-\frac{3}{2}}\int_{\R^3} e^{-i k\cdot x}u(x)dx.
$$
If there is no danger of confusion, we will also use the notation $\hat u=\F u$.
For vector fields this is understood to be applied componentwise.

We next consider the elastic energy. We directly discuss this for general affine boundary data $Fx + b$ for some $F \in \conv\{e^{(1)},e^{(2)}, e^{(3)}\}$ and $b\in \R^3$. In this case it reads
\begin{equation}\label{eq:elastic-energy}
E_{el}(\chi)=\inf_{u \in \mathcal{A}_F} \int_{\Omega}\big|e(\nabla u)-\chi\big|^2dx,
\end{equation}
with $\mathcal{A}_F$ as in \eqref{eq:BCset}.
Seeking to use Fourier methods, we extend the expression for the elastic energy to an energy defined in the whole space: Considering $u(x) = v(x) + Fx +b$, we obtain
\begin{align}
\label{eq:el_freq}
\begin{split}
E_{el}(\chi)&=\inf_{u \in \mathcal{A}_F} \int_{\Omega}\big|e(\nabla u)-\chi\big|^2dx
= \inf\limits_{v \in H^{1}_0(\Omega;\R^3)} \int_{\Omega}\big|e(\nabla v)-(\chi-F)\big|^2dx\\
&\geq \inf\limits_{v \in H^{1}(\R^3;\R^3)} \int_{\R^{3}}\big|e(\nabla v)-\tilde{\chi}\big|^2dx.
\end{split}
\end{align}
Here we have set 
\begin{align}
\label{eq:tildechi}
\tilde{\chi}:= \left\{
\begin{array}{ll}
\chi -F \mbox{ in } \Omega,\\
0 \mbox{ in } \R^3 \setminus \overline{\Omega}.
\end{array}
\right.
\end{align}
In what follows, with slight abuse of notation, in cases in which there is no danger of confusion, we will often drop the tilda in the notation for $\tilde{\chi}$.

Due to the extension to a whole space problem, we can immediately infer a lower bound for the elastic energy in Fourier space by invoking \cite[Lemma 4.1]{KKO13}.

\begin{lem}[Lemma 4.1 in \cite{KKO13}]
\label{lem:KKO41}
Let $F\in\conv\{e^{(1)},e^{(2)},e^{(3)}\}$, $E_{el}$ be defined by \eqref{eq:elastic-energy}, let $\chi\in L^{\infty}(\Omega;K)$ 
and let $\tilde\chi$ be as in \eqref{eq:tildechi}. Then it holds
\begin{equation*}
E_{el}(\chi) \gtrsim \int_{\R^3}\F \tilde{\chi}_{\rm diag}(k) M(k) \overline{(\F\tilde{\chi}_{\rm diag}(k))^t}dk,
\end{equation*}
where, for every $k\neq0$, $M(k)$ is the symmetric and positive semidefinite matrix given by
\begin{equation*}
M(k):=\frac{1}{|k|^4}
\begin{pmatrix}
(k_2^2+k_3^2)^2 & k_1^2 k_2^2 & k_1^2 k_3^2 \\
k_2^2 k_1^2 & (k_1^2+k_3^2)^2 & k_2^2 k_3^2 \\
k_1^2 k_3^2 & k_2^2 k_3^2 & (k_1^2+k_2^2)^2
\end{pmatrix},
\end{equation*}
and where $\tilde\chi_{\rm diag}:\R^3\to\R^3$ is $\tilde\chi_{\rm diag}=(\tilde\chi_{11},\tilde\chi_{22},\tilde\chi_{33})$.
\end{lem}

The lower bound from \eqref{eq:el_freq} can now be rephrased in terms of the vanishing of the Fourier multiplier from Lemma \ref{lem:KKO41}.
Again, this coercivity estimate comes from \cite[Lemma 4.2]{KKO13}.

\begin{lem}[Lemma 4.2 in \cite{KKO13}]
Let $F\in\conv\{e^{(1)},e^{(2)},e^{(3)}\}$, $E_{el}$ be defined by \eqref{eq:elastic-energy}, and let $\chi\in L^{\infty}(\Omega;K)$ and let $\tilde{\chi} \in BV(\R^3;\R^3)$ be as in \eqref{eq:tildechi}. Then it holds
\begin{equation*}
E_{el}(\chi) \gtrsim \sum_{j=1}^3\int_{\R^3}m_j(k)|\F\tilde{\chi}_{jj}(k)|^2 dk,
\end{equation*}
where, for every $j\in\{1,2,3\}$, $m_j$ is defined as the conical multiplier
\begin{equation}
\label{eq:mj}
m_j(k)=\dist^2\Big(\frac{k}{|k|},\B_j\Big)
\end{equation}
with $\B_j$ as in \eqref{eq:Bij}.
\end{lem}

Before proceeding with our analysis, we first introduce our notation of Fourier multipliers. A function $m\in L^\infty(\R^3)$ gives rise to a \emph{Fourier multiplier} $m(D)$ defined as 
\begin{align}
\label{eq:mult}
\F(m(D)u)(k):= m(k) \F u(k).
\end{align}
For further use, we recall a corollary of the Marcinkiewicz multiplier theorem on $\R^n$ (see, for instance, \cite[Corollary 6.2.5]{Grafakos}) which provides $L^p$-$L^p$ bounds of regular Fourier multipliers provided a suitable decay of their derivatives holds.
We further recall that, if $m(k)=m(-k)$ and $u$ is a real-valued function then also $m(D)u$ is real valued.
\begin{prop}\label{prop:grafakos}
Let $m$ be a bounded $C^\infty(\R^n \setminus \{0\})$ function.
Assume that for all $h\in\{1,\dots,n\}$, all distinct $j_1,\dots, j_h\in\{1,\dots,n\}$, and all $k_j\in\R\setminus\{0\}$ with $j\in\{j_1,\dots,j_h\}$ we have
\begin{equation}\label{eq:decay-multiplier}
|\p_{j_1}\dots\p_{j_h} m(k)|\le A|k_{j_1}|^{-1}\dots|k_{j_h}|^{-1}
\end{equation}
for some $A>0$.
Then for all $p\in(1,\infty)$, there exists a constant $C_n>0$ depending on the dimension such that for every $u\in L^p(\R^n;\C)$ it holds
$$
\|m(D) u\|_{L^p}\le\|u\|_{L^p} C_n(A+\|m\|_{L^\infty})\max\Big\{p,\frac{1}{p-1}\Big\}^{6n}.
$$
\end{prop}

\subsection{Localization results in Fourier space}

Following \cite{KW16,RT21,RT22} and the related literature, we localize our phase indicator in some truncated cones in Fourier space, with the excess $L^2$ mass (outside these truncated cones) being quantified by the total energy rescaled appropriately according to the size of the truncated cones.

For $j\in\{1,2,3\}$ let $m_j$ be as in \eqref{eq:mj} and define $C_{j,\mu,\mu_2}:=\{k\in\R^3 : m_j(k)<\mu^2, |k|<\mu_2\}$. We note that each truncated cone $C_{j,\mu,\mu_2}$ is in turn a union of four truncated cones in the sense that these are the union of 8 connected truncated cones if we remove the vertex.
Analogously, we also define for every $j\in\{1,2,3\}$ and $b\in\B_j$ the truncated cones around a single lamination direction $C_{b,\mu,\mu_2}=\{k\in\R^3 : \frac{|k|^2-|k\cdot b|^2}{|k|^2}<\mu^2, |k|<\mu_2\}$.
We then define the corresponding Fourier multipliers as $\chi_{j,\mu,\mu_2}(D) v=\F^{-1}(\eta_{C_{j,\mu,\mu_2}}\hat v)$, for every $v\in L^2(\R^3)$, where $\eta_{C_{j,\mu,\mu_2}}$ is a \emph{smoothed-out} even characteristic function of the set $C_{j,\mu,\mu_2}$ complying with the decay properties from \eqref{eq:decay-multiplier} in a suitable coordinate system.
Analogous notation is used for $\chi_{b,\mu,\mu_2}(D)$.
Arguing as in \cite{RT21,RT22, RT23}, we obtain the following rough first Fourier localization bound:

\begin{lem}[Lemma 4.5, \cite{RT22}]
\label{lem:loc0}
Let $F\in\conv\{e^{(1)},e^{(2)},e^{(3)}\}$, $E_{el}$ and $E_{surf}$ be defined by \eqref{eq:elastic-energy} and \eqref{eq:total-energy}, respectively. Let $\chi\in BV(\Omega;K)$ and let $\tilde{\chi}$ be as in \eqref{eq:tildechi}.
Then, for every $\mu\in(0,1)$ and $\mu_2>0$ we have
\begin{equation*}
\sum_{j=1}^3\|\tilde{\chi}_{jj}-\chi_{j,\mu,\mu_2}(D)\tilde{\chi}_{jj}\|_{L^2}^2 \lesssim \mu^{-2}E_{el}(\chi)+\mu_2^{-1}(E_{surf}(\chi) + \Per(\Omega)).
\end{equation*}
\end{lem}

These estimates encode the high frequency and coercivity bounds from the surface and elastic energies.

Next, we deduce low frequency bounds, similar in spirit as in \cite{RT21,RRT23,KW16}.
Here we use a particularly convenient form of these bounds which originates from a more general auxiliary result in \cite{RRTT23}. It can be viewed as a continuum version of estimates which had been used in the periodic setting in \cite {RT22}.

\begin{lem}[Proposition 3.1(i) \cite{RRTT23}]\label{lem:low-freq}
Let $\Omega \subset \R^n$ be an open, bounded set.
Let $f\in L^2(\R^3)$  with $\supp(f) \subset \Omega$ and $b\in\Sp^2$.
Then for every $\nu>0$ it holds
$$
\int_{|k|^2-|k\cdot b|^2\le\nu^2}|\hat f|^2 dk \lesssim \nu^2 \int_{\R^3}|\hat f|^2 dk.
$$
\end{lem}

\begin{proof}
For convenience, we recall the short proof. To this end, we set $v_1:=b$ and extend this to an orthonormal basis $v_1,v_2,v_3$. With slight abuse of notation, we refer to the coordinates in this basis as $k_1:=k \cdot v_1, k_2:=k\cdot v_2, k_3:= k \cdot v_3$. Then,
\begin{align*}
\int_{|k|^2-|k\cdot b|^2\le\nu^2}|\hat f|^2 dk 
&= \int_{k_2^2 + k_3^2\le\nu^2}|\hat f|^2 dk  \lesssim \nu^2 \int\limits_{\R}  \Big(\sup\limits_{(k_2, k_3) \in \R^2} |\hat f(k_1,k_2,k_3)|\Big)^2 d k_1\\
&\lesssim \nu^2 \int\limits_{\R} \Big(\int\limits_{\R^2} |\mathcal{F}_{k_1} f(k_1,x_2,x_3)|  dx_2 dx_3\Big)^2 dk_1\\
&\lesssim \diam(\Omega) \nu^2 \int\limits_{\R} \int\limits_{\R^2} |\mathcal{F}_{k_1} f(k_1,x_2,x_3) |^2 d k_1 dx_2 dx_3
\lesssim \nu^2 \int_{\R^3}|\hat f|^2 dk.
\end{align*}
Here we have used the notation $\mathcal{F}_{k_1} f$ for the Fourier transform only in $k_1$ and invoked the $L^{\infty}$-$L^{1}$ estimate for the Fourier transform together with Hölder's inequality and the fact that $\diam(\Omega)$ is bounded. 
\end{proof}

\subsection{Reducing the relevant frequency region}
Finally, with the previous preparatory results in hand, we turn to a second frequency localization argument which is valid for second order laminates and provides the key step in the proof of Theorems \ref{thm:main_zero} and \ref{thm:main_general}.

\begin{prop}[A second conical localization argument] 
\label{prop:loc1}
Let $F\in\conv\{e^{(1)},e^{(2)},e^{(3)}\}$, $\chi \in BV(\Omega;K)$, $\tilde{\chi}$ as in \eqref{eq:tildechi}, $E_{el}$, $E_{surf}$ as in \eqref{eq:elastic-energy} and \eqref{eq:total-energy}, respectively.
There exists a constant $\mu_0=\mu_0(\B)\in(0,1)$ such that for every $\mu\in(0,\mu_0)$,
$\mu_2 > \mu_3>0$, there holds
\begin{align}
\label{eq:cone-reduction1}
\begin{split}
&\Big|\int_{\R^3}\tilde{\chi}_{11}\tilde{\chi}_{22}\tilde{\chi}_{33}dx \Big| \\
&\lesssim \sum_{\substack{(b_1, b_2, b_3)\in \tilde{\B}}}\Big|\int_{\R^3}(\chi_{b_1,\mu,[\mu_3,\mu_2]}(D)\tilde\chi_{11})(\chi_{b_2,\mu,[\mu_3,\mu_2]}(D)\tilde\chi_{22})(\chi_{b_3,\mu,[\mu_3,\mu_2]}(D)\tilde\chi_{33})dx \Big| \\
& \quad + (\mu^{-2}E_{el}(\chi)+\mu_2^{-1}(E_{surf}(\chi) + \Per(\Omega)))^{\frac{1}{2}} + \mu\mu_3\|\tilde\chi\|_{L^2}.
\end{split}
\end{align}
Here $\chi_{b_j,\mu,[\mu_3,\mu_2]}(D):= \chi_{b_j,\mu, \mu_2}(D)-\chi_{b_j,\mu,\mu_3}(D)$ for $j\in\{1,2,3\}$ and 
\begin{align}
\label{eq:index}
\tilde\B:=\left\{(b_1,b_2,b_3) \in \B_1\times\B_2\times\B_3 : b_1,b_2\in\B_1\cap\B_2 \text{ or } b_1,b_3\in\B_1\cap\B_3 \text{ or } b_2,b_3\in\B_2\cap\B_3
\right\}.
\end{align}
\end{prop}

In what follows the trilinear quantity on the left hand side of inequality \eqref{eq:cone-reduction1} which had first been identified in \cite{KKO13} will play a major role in encoding non-trivial information for laminates of second order, c.f. Remark \ref{rmk:trilin}.

\begin{proof}
We split the proof into three parts.
In the first two steps we write each state $\tilde\chi_{jj}$ as a sum of functions which contain information on the oscillations in a specific lamination direction $b_j\in\B_j$, following the principles in \cite[Section 4.3]{KKO13}.
The third and final step is then devoted to the proof of the result.

\emph{Step 1: Decomposition into  the most relevant lamination directions.}
Let $\{\tilde\eta_b\}_{b\in\B}$ be a partition of unity of $\Sp^2$, with $\tilde\eta_{b}(-\hat k)=\tilde\eta_{b}(\hat k)$ and  $\tilde\eta_b:\Sp^2\to[0,1]$ be $C^\infty$ functions being equal to $1$ in a neighbourhood of $b$.
Denote then $\eta_b(k):=\tilde\eta_b(\frac{k}{|k|})$ and define
\begin{equation}\label{eq:lam-dec}
f_{j,b}:=\eta_b(D)\tilde\chi_{jj}.
\end{equation}
These functions satisfy
\begin{equation*}
\begin{split}
	\tilde\chi_{11} &= f_{1,b_{12}}+f_{1,b_{21}}+f_{1,b_{13}}+f_{1,b_{31}}+f_{1,b_{23}}+f_{1,b_{32}}, \\
	\tilde\chi_{22} &= f_{2,b_{12}}+f_{2,b_{21}}+f_{2,b_{13}}+f_{2,b_{31}}+f_{2,b_{23}}+f_{2,b_{32}}, \\
	\tilde\chi_{33} &= f_{3,b_{12}}+f_{3,b_{21}}+f_{3,b_{13}}+f_{3,b_{31}}+f_{3,b_{23}}+f_{3,b_{32}},		
\end{split}
\end{equation*}
and, by Proposition \ref{prop:grafakos}, for any $p>1$ it holds $\|f_{j,b_j}\|_{L^p}\lesssim 1$ for every $j\in\{1,2,3\}$, $b_j\in\B_j$.

Using that by the vanishing trace constraint for the matrices $e^{(j)}$, we have $\tilde{\chi}_{11} + \tilde{\chi}_{22} + \tilde{\chi}_{33} = 0$,
we further observe that we can reduce the system above to the following decomposition
\begin{equation}\label{eq:lam-dec2}
\begin{split}
	\tilde\chi_{11} &= f_{b_{12}}+f_{b_{21}}-f_{b_{13}}-f_{b_{31}}, \\
	\tilde\chi_{22} &= -f_{b_{12}}-f_{b_{21}}+f_{b_{23}}+f_{b_{32}}, \\
	\tilde\chi_{33} &= f_{b_{13}}+f_{b_{31}}-f_{b_{23}}-f_{b_{32}},
\end{split}
\end{equation}
where
\begin{eqnarray*}\label{eq:lam-dec3}
	&f_{b_{12}} := f_{1,b_{12}}-f_{2,b_{31}}, &f_{b_{21}} := f_{1,b_{21}}-f_{2,b_{13}},\\
	&f_{b_{23}} := f_{2,b_{23}}-f_{3,b_{12}}, &f_{b_{32}} := f_{2,b_{32}}-f_{3,b_{21}}, \\
	&f_{b_{31}} := f_{3,b_{31}}-f_{1,b_{23}}, &f_{b_{13}} := f_{3,b_{13}}-f_{1,b_{32}},
\end{eqnarray*}
see also \cite[Lemma 4.3 and Proposition 4.4]{KKO13}.

Moreover, straightforward applications of Lemmas \ref{lem:loc0} and \ref{lem:low-freq} yield that for any $b\in\B$
\begin{equation}\label{cor:loc0}
\|f_{b}-\chi_{b,\mu,\mu_2}(D)f_{b}\|_{L^2}^2 \lesssim (\mu^{-2}+1)E_{el}(\chi)+\mu_2^{-1}(E_{surf}(\chi) + \Per(\Omega)).
\end{equation}
and
\begin{equation}\label{eq:low-freq}
\int_{|k|^2-|k\cdot b|^2\le\nu^2}|\hat f_{b}|^2 dk \lesssim \nu^2,
\end{equation}
respectively.
For later use, we notice that we can find a constant $\mu_0>0$ depending on $\B$ such that $\chi_{b,\mu,\mu_2}<\eta_b$ and $\eta_{b'}\chi_{b,\mu,\mu_2}=0$ for every $b,b'\in\B$, $b'\neq b$, $0<\mu < \mu_0$.

\emph{Step 2: Cancellations.}
We first note that, by the decomposition \eqref{eq:lam-dec2} we have
\begin{align}
\label{eq:sum_aux}
\Big|\int_{\R^3}\tilde\chi_{11}\tilde\chi_{22}\tilde\chi_{33} dx\Big| = \Big|\sum_{(b_1,b_2,b_3)\in\B_1\times\B_2\times\B_3}\int_{\R^3}f_{b_1}^{(1)}f_{b_2}^{(2)}f_{b_3}^{(3)} dx\Big|,
\end{align}
where we denoted by $f_b^{(j)}$ the function $f_b$ with the corresponding sign which it carries in the decomposition \eqref{eq:lam-dec2} of the state $\tilde\chi_{jj}$.
Using this, we next observe that the right-hand side of \eqref{eq:sum_aux} can be simplified via some cancellations.
Indeed, we observe that for any triple product $f_{b_1}^{(1)}f_{b_2}^{(2)}f_{b_3}^{(3)}$ with $(b_1,b_2,b_3)\in(\B_1\cap\B_2)\times(\B_2\cap\B_3)\times(\B_1\cap\B_3)$ the sum \eqref{eq:sum_aux} also contains a triple of the form $f_{b_3}^{(1)}f_{b_1}^{(2)}f_{b_2}^{(3)}=-f_{b_1}^{(1)}f_{b_2}^{(2)}f_{b_3}^{(3)}$. These contributions hence cancel out.
Recalling the definition of $\tilde{\mathcal{B}}$ from \eqref{eq:index}, the sum from \eqref{eq:sum_aux} can hence be rewritten and estimated as
\begin{align}
\label{eq:step1bound}
\Big|\int_{\R^3}\tilde\chi_{11}\tilde\chi_{22}\tilde\chi_{33} dx\Big| 
&= \Big|\sum_{(b_1,b_2,b_3)\in\tilde\B}\int_{\R^3}f_{b_1}^{(1)}f_{b_2}^{(2)}f_{b_3}^{(3)} dx\Big|
 \leq \sum_{(b_1,b_2,b_3)\in\tilde\B} \Big|\int_{\R^3}f_{b_1}^{(1)}f_{b_2}^{(2)}f_{b_3}^{(3)} dx\Big|.
\end{align}

\emph{Step 3: Commutator arguments.}
Now, using  iterated triangle inequalities, we claim that it suffices to study the expression from the right-hand side \eqref{eq:step1bound} localized to the relevant truncated cones. As we use no further cancellation effects for which we need to track the superscript in the notation $f_{b_j}^{(k)}$, in what follows below we drop the superscript. 
For any fixed triple $(b_1,b_2,b_3)\in \tilde\B$, applying first the triangle inequality and then H\"older's inequality, we infer that

\begin{align}
\label{eq:triangle1}
\begin{split}
\Big|\int_{\R^3}f_{b_1}f_{b_2}f_{b_3} dx \Big| &\le \Big|\int_{\R^3}(f_{b_1}-\chi_{b_1,\mu,\mu_2}(D)f_{b_1})f_{b_2}f_{b_3}dx \Big| \\
& \qquad +\Big|\int_{\R^3}(\chi_{b_1,\mu,\mu_2}(D)f_{b_1})f_{b_2}f_{b_3}dx \Big| \\
& \le \|f_{b_1}-\chi_{b_1,\mu,\mu_2}(D)f_{b_1}\|_{L^2}\|f_{b_2}\|_{L^4}\|f_{b_3}\|_{L^4} \\
& \qquad + \Big|\int_{\R^3}(\chi_{b_1,\mu,\mu_2}(D)f_{b_1})f_{b_2}f_{b_3}dx \Big|\\
& \leq C((\mu^{-2}+1)E_{el}(\chi) + \mu_2^{-1}(E_{surf}(\chi) + \Per(\Omega)))^{\frac{1}{2}}\\
&\qquad + \Big|\int_{\R^3}(\chi_{b_1,\mu,\mu_2}(D)f_{b_1})f_{b_2}f_{b_3}dx \Big|.
\end{split}
\end{align}

Here in the last estimate we used the energy estimates from \eqref{cor:loc0}. 
We thus focus on the second term on the right-hand side of \eqref{eq:triangle1}. For this we obtain

\begin{align}
\label{eq:triangle2}
\begin{split}
\Big|\int_{\R^3}(\chi_{b_1,\mu,\mu_2}(D)f_{b_1})f_{b_2}f_{b_3}dx \Big| &\le \Big|\int_{\R^3}(\chi_{b_1,\mu,\mu_2}(D)f_{b_1})(f_{b_2}-\chi_{b_2,\mu,\mu_2}(D)f_{b_2})f_{b_3}dx \Big| \\
& \qquad +\Big|\int_{\R^3}(\chi_{b_1,\mu,\mu_2}(D)f_{b_1})(\chi_{b_2,\mu,\mu_2}(D)f_{b_2})f_{b_3}dx \Big| \\
& \le \|f_{b_2}-\chi_{b_2,\mu,\mu_2}(D)f_{b_2}\|_{L^2}\|(\chi_{b_1,\mu,\mu_2}(D)f_{b_1})f_{b_3}\|_{L^2} \\
& \qquad + \Big|\int_{\R^3}(\chi_{b_1,\mu,\mu_2}(D)f_{b_1})(\chi_{b_2,\mu,\mu_2}(D)f_{b_2})f_{b_3}dx \Big| \\
& \le \|f_{b_2}-\chi_{b_2,\mu,\mu_2}(D)f_{b_2}\|_{L^2}\|\chi_{b_1,\mu,\mu_2}(D)f_{b_1}\|_{L^4}\|f_{b_3}\|_{L^4} \\
& \qquad + \Big|\int_{\R^3}(\chi_{b_1,\mu,\mu_2}(D)f_{b_1})(\chi_{b_2,\mu,\mu_2}(D)f_{b_2})f_{b_3}dx \Big|\\
& \lesssim \|f_{b_2}-\chi_{b_2,\mu,\mu_2}(D)f_{b_2}\|_{L^2}\|f_{b_1}\|_{L^4}\|f_{b_3}\|_{L^4} \\ 
& \qquad + \Big|\int_{\R^3}(\chi_{b_1,\mu,\mu_2}(D)f_{b_1})(\chi_{b_2,\mu,\mu_2}(D)f_{b_2})f_{b_3}dx \Big| \\
& \leq C((\mu^{-2}+1)E_{el}(\chi) + \mu_2^{-1}(E_{surf}(\chi) + \Per(\Omega)))^{\frac{1}{2}} \\
& \qquad + \Big|\int_{\R^3}(\chi_{b_1,\mu,\mu_2}(D)f_{b_1})(\chi_{b_2,\mu,\mu_2}(D)f_{b_2})f_{b_3}dx \Big|.
\end{split}
\end{align}

Again, in the last estimate, for the first term on the right-hand side we used the bound from estimate  \eqref{cor:loc0}
whereas in the second last line, we used Proposition \ref{prop:grafakos} to infer $\|\chi_{b_1,\mu,\mu_2}(D)f_{b_1}\|_{L^4}\lesssim\|f_{b_1}\|_{L^4}$.
As above, we thus only consider the second contribution on the right-hand side of \eqref{eq:triangle2}. 
Then,
\begin{align}
\label{eq:triangle3}
\begin{split}
&\Big|\int_{\R^3}(\chi_{b_1,\mu,\mu_2}(D)f_{b_1})(\chi_{b_2,\mu,\mu_2}(D)f_{b_2})f_{b_3}dx \Big| \\
& \quad \le \Big|\int_{\R^3}(\chi_{b_1,\mu,\mu_2}(D)f_{b_1})(\chi_{b_2,\mu,\mu_2}(D)f_{b_2})(f_{b_3}-\chi_{b_3,\mu,\mu_2}(D)f_{b_3})dx \Big| \\
& \qquad +\Big|\int_{\R^3}(\chi_{b_1,\mu,\mu_2}(D)f_{b_1})(\chi_{b_2,\mu,\mu_2}(D)f_{b_2})(\chi_{b_3,\mu,\mu_2}(D)f_{b_3})dx \Big| \\
& \quad \le \|(\chi_{b_1,\mu,\mu_2}(D)f_{b_1})(\chi_{b_2,\mu,\mu_2}(D)f_{b_2})\|_{L^2}\|f_{b_3}-\chi_{b_3,\mu,\mu_2}(D)f_{b_3}\|_{L^2} \\
& \qquad +\Big|\int_{\R^3}(\chi_{b_1,\mu,\mu_2}(D)f_{b_1})(\chi_{b_2,\mu,\mu_2}(D)f_{b_2})(\chi_{b_3,\mu,\mu_2}(D)f_{b_3})dx\Big| \\
& \quad \le \|\chi_{b_1,\mu,\mu_2}(D)f_{b_1}\|_{L^4}\|\chi_{b_2,\mu,\mu_2}(D)f_{b_2}\|_{L^4}\|f_{b_3}-\chi_{b_3,\mu,\mu_2}(D)f_{b_3}\|_{L^2} \\
& \qquad +\Big|\int_{\R^3}(\chi_{b_1,\mu,\mu_2}(D)f_{b_1})(\chi_{b_2,\mu,\mu_2}(D)f_{b_2})(\chi_{b_3,\mu,\mu_2}(D)f_{b_3})dx \Big| \\
& \quad \lesssim 
\|f_{b_3}-\chi_{b_3,\mu,\mu_2}(D)f_{b_3}\|_{L^2}\|f_{b_1}\|_{L^4}\|f_{b_2}\|_{L^4} \\& \qquad + \Big|\int_{\R^3}(\chi_{b_1,\mu,\mu_2}(D)f_{b_1})(\chi_{b_2,\mu,\mu_2}(D)f_{b_2})(\chi_{b_3,\mu,\mu_2}(D)f_{b_3})dx \Big| \\
& \quad \le C((\mu^{-2}+1)E_{el}(\chi) + \mu_2^{-1}(E_{surf}(\chi) + \Per(\Omega)))^{\frac{1}{2}} \\
& \qquad +\Big|\int_{\R^3}(\chi_{b_1,\mu,\mu_2}(D)\tilde\chi_{11})(\chi_{b_2,\mu,\mu_2}(D)\tilde\chi_{22})(\chi_{b_3,\mu,\mu_2}(D)\tilde\chi_{33})dx \Big|.
\end{split}
\end{align}
In the second last line, we again used 
Proposition \ref{prop:grafakos} to get $\|\chi_{b_j,\mu,\mu_2}(D)f_{b_j}\|_{L^4}\lesssim\|f_{b_j}\|_{L^4}$.
In the last line, we exploited the energy estimates \eqref{cor:loc0} and used the fact that $\chi_{b_j,\mu,\mu_2}(D)f_{b_j}=\chi_{b_j,\mu,\mu_2}(D)\tilde\chi_{jj}$ for every $j\in\{1,2,3\}$ and $b_{j}\in\B_j$ (cf.\ Step 1).
It thus remains to prove that 
\begin{equation}
\label{eq:triangle4}
\begin{split}
&\Big|\int_{\R^3}(\chi_{b_1,\mu,\mu_2}(D)\tilde\chi_{11})(\chi_{b_2,\mu,\mu_2}(D\tilde\chi_{22})(\chi_{b_3,\mu,\mu_2}(D)\tilde\chi_{33})dx \Big|\\
&\leq \Big|\int_{\R^3}(\chi_{b_1,\mu,[\mu_3,\mu_2]}(D)\tilde\chi_{11})(\chi_{b_2,\mu,[\mu_3,\mu_2]}(D)\tilde\chi_{22})(\chi_{b_3,\mu,[\mu_3,\mu_2]}(D)\tilde\chi_{33})dx \Big| + C \mu\mu_3\|\tilde\chi\|_{L^2}.
\end{split}
\end{equation}
Indeed, this follows by using the low frequency bounds from the estimate \eqref{eq:low-freq} and noticing that $C_{b_j,\mu,\mu_3}\subset\{k\in\R^3 : |k|^2-|k\cdot b_j|^2\lesssim(\mu\mu_3)^2\}$.
Then, exploiting this together with the triangle inequality and Calder\'on-Zygmund estimates, similarly as in the arguments above,
we obtain \eqref{eq:triangle4}.
Finally, combining \eqref{eq:triangle1}--\eqref{eq:triangle4} and by carrying
out these estimates for every possible combination of vectors $(b_1,b_2,b_3)\in
\tilde\B$, we infer the desired result.
\end{proof}

Next, we seek to further reduce the relevant localization for the (truncated) cones. 
To do so we exploit the fact that we cannot have all the three states localized in the same (truncated) cone at the same time.

\begin{prop}[A final conical localization argument] 
\label{prop:loc2}
Let $F\in\conv\{e^{(1)},e^{(2)},e^{(3)}\}$, $\chi \in BV(\Omega;K)$, $\tilde{\chi}$ as in \eqref{eq:tildechi}, $E_{el}$, $E_{surf}$ as in \eqref{eq:elastic-energy} and \eqref{eq:total-energy} respectively.
Let $b_1\in\B_1$, $b_2\in\B_2$ and $b_3\in\B_3$ and assume that either
 $\{b_1,b_2,b_3\}\subset \R^3$ forms a basis or is such that $b_i \neq b_j = b_k$ for $i,j,k\in\{1,2,3\}$, $i\neq j$, $j\neq k$, $i\neq k$. Then,
 there exist constants $\mu_0=\mu_0(b_1,b_2,b_3)\in(0,1)$ and $M=M(b_1,b_2,b_3)>4$ such that for every $\mu\in(0,\mu_0)$, $\mu_2>0$ and $\mu_3 = M\mu \mu_2$  it holds
 \begin{align}
 \label{eq:low_freq_loc}
\Big|\int_{\R^3}(\chi_{b_1,\mu,[\mu_3,\mu_2]}(D)\tilde{\chi}_{11})(\chi_{b_2,\mu,[\mu_3,\mu_2]}(D)\tilde{\chi}_{22})(\chi_{b_3,\mu,[\mu_3,\mu_2]}(D)\tilde{\chi}_{33})dx\Big|  =0.
 \end{align}
\end{prop}

\begin{rmk}
\label{rmk:combis}
After having exploited the cancellation of the contributions in $\mathcal{B} \setminus \tilde{\mathcal{B}}$ in Step 3 of Proposition \ref{prop:loc1}, in the discussion in Proposition \ref{prop:loc2} we only need to consider vectors $(b_1,b_2,b_3) \in \tilde{\mathcal{B}}$.
We point out that, due to the structure of the sets $\B_j$, for $b_1\in\B_1$, $b_2\in\B_2$ and $b_3\in\B_3$ for vectors in such that $(b_1, b_2, b_3) \in \tilde{\mathcal{B}}$ only the following two cases can occur: either
\begin{itemize}
 \item $\{b_1,b_2,b_3\}\subset \R^3$ forms a basis,
 \item or it is such that two of the three vectors agree.
 \end{itemize}
 Hence, the assumptions of Proposition \ref{prop:loc2} cover all possibilities of contributions which need to be estimated.
 
 The claim above can be proved as follows: let $(b_1,b_2,b_3)\in\tilde\B$ and
 let $i,j\in\{1,2,3\}$, $i\neq j$ such that
 $b_i,b_j\in\B_i\cap\B_j=\{b_{ij},b_{ji}\}$ and let $k \in \{1,2,3\}\setminus\{i,j\}$.
 If $b_i=b_j$ we fall in the second bullet point above.
 If $b_i\neq b_j$, they are orthogonal to $e_k$ (cf.\ \eqref{eq:normals}).
 Since $b_k\in\B_k$, $b_k\cdot e_k\neq0$ (see again \eqref{eq:normals}), thus $\{b_1,b_2,b_3\}$ is a basis and the claim is proved.
\end{rmk}

\begin{figure}[t]
\begin{center}
\includegraphics{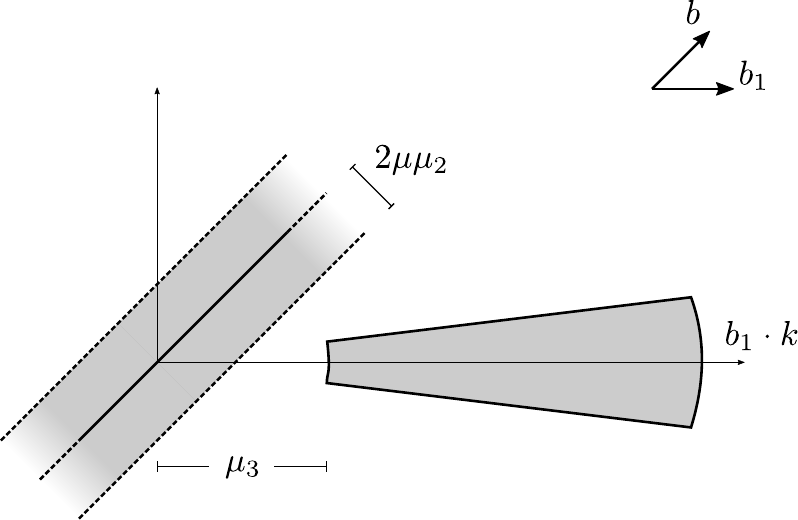}
\caption{A graphic representation of the arguments in the proof of Proposition \ref{prop:loc2}.
The vector $b$ is the projection of $b_1$ on the space $V$.}
\label{fig:cones}
\end{center}
\end{figure}

\begin{proof}[Proof of Proposition \ref{prop:loc2}]
In both the cases we can apply a bootstrap argument.
Indeed, Plancherel's theorem and Hölder's inequality yield
\begin{align}
\label{eq:case1}
\begin{split}
&\Big|\int_{\R^3} (\chi_{b_1,\mu,[\mu_3,\mu_2]}(D) \tilde\chi_{11})\left[(\chi_{b_2,\mu,[\mu_3,\mu_2]}(D)\tilde\chi_{22})(\chi_{b_3,\mu,[\mu_3,\mu_2]}(D)\tilde\chi_{33})\right] dx \Big| \\
& \quad = \Big|\int_{\R^3}\tilde\chi_{11}\chi_{b_1,\mu,[\mu_3,\mu_2]}(D)\left[(\chi_{b_2,\mu,[\mu_3,\mu_2]}(D)\tilde\chi_{22})(\chi_{b_3,\mu,[\mu_3,\mu_2]}(D)\tilde\chi_{33}) \right]dx \Big| \\
& \quad \lesssim\|\chi_{b_1,\mu,[\mu_3,\mu_2]}(D)(\chi_{b_2,\mu,[\mu_3,\mu_2]}(D)\tilde\chi_{22})(\chi_{b_3,\mu,[\mu_3,\mu_2]}(D)\tilde\chi_{33})\|_{L^2} \|\tilde{\chi}_{11}\|_{L^2}\\
& \quad \lesssim \Big|\int_{C_{b_1,\mu,\mu_2}\setminus C_{b_1,\mu,\mu_3}}\big|(\chi_{C_{b_2,\mu,[\mu_3,\mu_2]}}(k)\mathcal{F}\tilde\chi_{22})*(\chi_{C_{b_3,\mu,[\mu_3,\mu_2]}}(k)\mathcal{F}\tilde\chi_{33})\big|^2d k\Big|^\frac{1}{2}.
\end{split}
\end{align}
Let $V:=\text{span}(b_2,b_3)$.
By definition, $C_{b_2,\mu,\mu_2},C_{b_3,\mu,\mu_2}\subset\{k\in\R^3 : \dist(k,V)\le\mu\mu_2\}$, thus
\begin{equation}\label{eq:cones-mink1}
C_{b_2,\mu,\mu_2}+C_{b_3,\mu,\mu_2}\subset\{k\in\R^3 : \dist(k,V)\le2\mu\mu_2\},
\end{equation}
in the sense of Minkowski sum.
Now since $b_1$ is independent of the linear space $V$, we have that
\begin{equation}\label{eq:cones-mink2}
C_{b_1,\mu,[\mu_3,\mu_2]}\subset\{k\in\R^3 : \dist(k,V)>2\mu\mu_2\}
\end{equation}
(see Figure \ref{fig:cones}) for every $\mu<\mu_0$ and $\mu_3:=M\mu\mu_2$ where $\mu_0=\mu_0(b_1,b_2,b_3)\in(0,1)$ and $M:=M(b_1,b_2,b_3)>4$.
Gathering \eqref{eq:cones-mink1} and \eqref{eq:cones-mink2} we obtain
$$
(C_{b_2,\mu,\mu_2}+C_{b_3,\mu,\mu_2})\cap C_{b_1,\mu,[\mu_3,\mu_2]}=\emptyset.
$$
Thus, the right hand side in \eqref{eq:case1} above vanishes.
 In concluding, we stress that the arguments above work independently of whether $V$ is a one or two-dimensional vector space. 
\end{proof}

\subsection{Proof of the lower bound estimates from Theorems \ref{thm:main_zero} and \ref{thm:main_general}}

With the previous results in hand, we begin by presenting the proof of the lower bound estimates in our main lower bound results, Theorems \ref{thm:main_zero} and \ref{thm:main_general}(i).

\begin{proof}[Proof of the lower bounds from Theorems \ref{thm:main_zero} and \ref{thm:main_general}(i)]
By Propositions \ref{prop:loc1} and \ref{prop:loc2},
combining \eqref{eq:cone-reduction1}, \eqref{eq:low_freq_loc}, Remark \ref{rmk:combis} and invoking the triangle inequality, we infer that
\begin{align}
\label{eq:lower_bd1}
\Big|\int_{\R^3}\tilde{\chi}_{11}\tilde{\chi}_{22}\tilde{\chi}_{33}dx \Big|^2  \lesssim (\mu\mu_3)^2+\mu^{-2}E_{el}(\chi)+\mu_2^{-1}(E_{surf}(\chi) + \Per(\Omega)),
\end{align}
for every $\mu\in(0,\mu_0)$, $\mu_2>0$ and $\mu_3=M\mu\mu_2$, where $\mu_0=\mu_0(\B)\in(0,1)$ and $M=M(\B)>4$.
We further note that by the properties of the cubic-to-tetragonal phase transformation it holds that
\begin{align*}
\Big|\int_{\R^3}\tilde{\chi}_{11}\tilde{\chi}_{22}\tilde{\chi}_{33}dx \Big|^2 \geq \min_{e\in K}\min_{j\in\{1,2,3\}}|e_{jj}-F_{jj}|^6 |\Omega|^2.
\end{align*}
We note that for $F\in \text{intconv}\{e^{(1)},e^{(2)}, e^{(3)}\}$, by the structure of the wells, we have that $\min_{e\in K}\min_{j\in\{1,2,3\}}|e_{jj}-F_{jj}|\gtrsim \dist(F, K^{(1)})>0$.

Therefore,

\begin{align}
\label{eq:lower_bd1}
\begin{split}
1 
&\lesssim (\mu\mu_3)^{2}+\mu^{-2}E_{el}(\chi)+\mu_2^{-1}(E_{surf}(\chi) + \Per(\Omega))\\
&\lesssim (\mu^2\mu_2)^2+\left(\mu^{-2}+\epsilon^{-1}\mu_2^{-1} \right) E_{\epsilon}(\chi) + \mu_2^{-1} \Per(\Omega).
\end{split}
\end{align}
In order to conclude the desired result, it remains to carry out an optimization argument: Choosing $\mu_2 \sim \mu^{-2}$ we can absorb the term $(\mu^2\mu_2)^2$ in the left-hand side, by possibly reducing the value of the constant $\mu_0$ if needed. 
Thus, 
$$
1 \lesssim (\mu_2 +(\mu_2\epsilon)^{-1})E_{\epsilon}(\chi) + \mu_2^{-1} \Per(\Omega).
$$
Optimizing and choosing $\mu_2\sim(\epsilon\mu_2)^{-1}$, we obtain $\mu_2\sim\epsilon^{-\frac{1}{2}}$ and therefore
$$
1\lesssim \epsilon^{-\frac{1}{2}}E_\epsilon(\chi) + \epsilon^{\frac{1}{2}} \Per(\Omega).
$$

Finally, choosing $\epsilon_0>0$ sufficiently small (depending only on $\Omega$ and $\dist(F,K^{(1)})$) then allows to absorb the perimeter contribution into the left hand side and hence yields the claim. 
\end{proof}

\begin{rmk}
\label{rmk:trilin}
We remark that, in general, for $F \in \text{conv}(K)\setminus \text{intconv}(K)$ the bound 
\begin{align*}
\min_{e\in K}\min\limits_{j \in \{1,2,3\}}|e_{jj}-F_{jj}|\gtrsim \dist(F, K)>0
\end{align*}
fails. Indeed, this is a component-wise bound, and for a simple laminate one of the components $\chi_{jj}$ of $\chi$ will essentially be equal to $F_{jj}$. Thus, the above argument no longer holds in this case and a different, less tight lower scaling bound will be deduced (corresponding to the different expected lower scaling behaviour for first order laminates).
\end{rmk}

We finally, conclude by presenting the proof of the lower bound in Theorem \ref{thm:main_general}(ii):

\begin{proof}[Proof of the lower bound in Theorem \ref{thm:main_general}(ii)]
The argument is direct and already follows from the first order localization argument from Proposition \ref{prop:loc1} together with the low frequency bound. Indeed, by Proposition \ref{prop:loc1} and the low frequency bounds we have
\begin{align*}
\sum\limits_{j =1}^3\|\tilde{\chi}_{jj}\|_{L^2}^2 
&\leq \sum\limits_{j =1}^3\|\tilde{\chi}_{jj}- \chi_{j,\mu,\mu_2}(D) \tilde{\chi}_{jj}\|_{L^2}^2 + \sum\limits_{j =1}^3\sum\limits_{b \in \mathcal{B}_j}\int\limits_{\{|k|^2-|b\cdot k|^2\leq (\mu\mu_2)^2\}}|\F\tilde{\chi}_{jj}|^2 dk\\
&\leq C(\mu\mu_2)^2+C(\mu^{-2} + \epsilon^{-1} \mu_2^{-1})E_{\epsilon}(\chi) + \mu_2^{-1} \Per(\Omega).
\end{align*}
Absorbing the term $(\mu\mu_2)^2$ in the left-hand side (yielding $\mu_2\sim\mu^{-1}$) and  optimizing in $\mu$ yields that $\mu\sim\epsilon^{\frac{1}{3}}$. Thus, by the triangle inequality,
\begin{align}
\label{eq:simple_lam}
|\Omega|\sum\limits_{j=1}^{3}\min_{e\in K}|e_{jj}-F_{jj}|^2
\leq \sum\limits_{j =1}^3 \|\tilde{ \chi}_{jj} \|_{L^2}^2
\leq C \Big( \epsilon^{-\frac{2}{3}}E_{\epsilon}(\chi) + \epsilon^\frac{1}{3} \Per(\Omega)\Big).
\end{align}
We note that for first order laminates
\begin{align*}
\min_{e\in K}\max\limits_{j\in\{1,2,3\}}|e_{jj}-F_{jj}|^2>0.
\end{align*}
Thus, multiplying \eqref{eq:simple_lam} by $\epsilon^{\frac{2}{3}}$ and choosing $\epsilon \in (0,\epsilon_0)$ for some $\epsilon_0>0$ small (depending on $\Per(\Omega)$ and $\dist(F,K)$ only) then allows us to absorb the perimeter contribution and to conclude the desired estimate.
\end{proof}

\begin{rmk}[Self-accommodation vs compatibility]
Our argument from above distinguishes the two relevant -- and competing -- physical phenomena of compatibility versus self-accommodation in a clear manner. Indeed, the localization bounds from Lemma \ref{lem:loc0} are clearly effects of compatibility (and result from the (anisotropic) coercivity properties of the energy). The more refined properties from Propositions \ref{prop:loc1} and \ref{prop:loc2} now already involve a combination of compatibility and self-accommodation in the sense that the upper bounds are purely effects of (in-)compatibility, while the relevance of the trilinear expression (i.e., its non-triviality) is a consequence of self-accommodation (or, in our case, the presence of Dirichlet data). More generally, the trilinear expression is only relevant as a lower bound in the case that one considers second order laminates as boundary conditions (see Remark \ref{rmk:trilin}). It is thus only in bounding the left hand expression in
\eqref{eq:lower_bd1} that the phenomenon of self-accommodation enters. 
\end{rmk}

\section{Upper-Bound Constructions: Proof of Theorems \ref{thm:upper}  and \ref{thm:upper2}}
\label{sec:upper}

In this section we present a construction involving two orders of lamination, attaining the energy scaling bound $\epsilon^\frac{1}{2}$ and thus proving the claimed upper bound from Theorem \ref{thm:upper}. Moreover, we discuss the simple laminate construction from Theorem \ref{thm:upper2}.

As outlined in the introduction, due to technical difficulties (explained below), the construction for Theorem \ref{thm:upper} does not attain Dirichlet boundary conditions on the whole boundary of the domain being instead constant in one direction. In spite of our construction thus not fully matching the lower bound setting from Theorems \ref{thm:main_zero} and \ref{thm:main_general}, the main features of a branched double laminate are present. We expect that the same scaling behaviour emerges when attaining a full Dirichlet boundary condition.

\smallskip

As the main part of this section, we present the construction for Theorem \ref{thm:upper} in Sections \ref{sec:features}-\ref{sec:final}. In Section \ref{sec:Thm4} we outline the proof of Theorem \ref{thm:upper2} for which we invoke the ideas from \cite[Proposition 6.3]{RT21}.

\subsection{Features and difficulties of the construction for Theorem \ref{thm:upper}} 
\label{sec:features}
Before turning to the precise technical aspects of the upper bound construction for Theorem \ref{thm:upper}, let us discuss the main features of it: As a main characteristic the construction discussed below presents \emph{two} levels of branched lamination (i.e., branching within branching) occurring at two different length scales.

The coarsest lamination arises between two ``auxiliary states'' (i.e., mixtures of two martensite states) -- the matrices $e^{(A)}, e^{(B)}$ from the introduction (see also Sections \ref{sec:grads} and \ref{sec:first}).
At this level, for the setting of the full Dirichlet data the main difficulty is that of defining a self-similar construction refining towards the boundary of a three-dimensional domain.
This problem has been overcome by the (third-order laminates) construction given in \cite[Section 6.2]{RT21}.
Nonetheless, using the construction from \cite[Section 6.2]{RT21} for the first lamination level in our present case gives rise to new technical difficulties when defining the second order ones.
These additional difficulties (which were not present in \cite[Section 6.2]{RT21}) are due to the fact that on the second lamination level (in which we split \emph{both} $e^{(A)}$ and $e^{(B)}$ into finer scale branched laminates) we have to compound two different finer branched laminates  ``at the same time'' (see Section \ref{sec:second}).

In order to avoid these difficulties, we will work with mixed Dirichlet and periodic data in what follows.
In this setting we will thus first define an essentially two-dimensional first order branched lamination construction (at the expense of not obtaining Dirichlet data on the full boundary). In a second step, we will then replace the two auxiliary states by a finer (second order) branched lamination between these affine stress-free states.

\subsection{Choice of the gradients}
\label{sec:grads}

As outlined in the introduction, one possible choice for attaining the zero boundary condition is by defining the two auxiliary states
$$
e^{(A)}:=\frac{1}{3}e^{(1)}+\frac{2}{3}e^{(2)}, \quad e^{(B)}:=\frac{1}{3}e^{(1)}+\frac{2}{3}e^{(3)}.
$$
The austenite matrix then is a second order laminate involving only wells in $\{e^{(1)},e^{(2)},e^{(3)}\}$:
$$
\textbf{0}=\frac{1}{2}e^{(A)}+\frac{1}{2}e^{(B)}.
$$

We choose gradients for the strains $e^{(j)}$ and for the auxiliary strains $e^{(A)}$ and $e^{(B)}$ such that the lamination between $e^{(A)}$ and $e^{(B)}$ occurs in direction $b_{32}$, and the finer lamination between $e^{(1)}$ and $e^{(2)}$ (respectively, $e^{(1)}$ and $e^{(3)}$) happens in direction $b_{21}$ (respectively, $b_{13}$).
This is the same setting as the one chosen in the construction from \cite[Section 6]{KKO13}. As a consequence, we will also consider the choice of gradients therein, i.e., we set
$$
A_1:=\begin{pmatrix}-2&2&0\\-2&1&1\\0&-1&1\end{pmatrix},
\quad
A_2:=\begin{pmatrix}1&-1&0\\1&-2&1\\0&-1&1\end{pmatrix},
\quad
A:=\begin{pmatrix}0&0&0\\0&-1&1\\0&-1&1\end{pmatrix},
$$
and
$$
B_1:=\begin{pmatrix}-2&0&2\\0&1&-1\\-2&1&1\end{pmatrix},
\quad
B_3:=\begin{pmatrix}1&0&-1\\0&1&-1\\1&1&-2\end{pmatrix},
\quad
B:=\begin{pmatrix}0&0&0\\0&1&-1\\0&1&-1\end{pmatrix}.
$$
With this choice, it can be easily checked that $e(A_j)=e^{(j)}$, $e(A)=e^{(A)}$, $e(B_j)=e^{(j)}$, $e(B)=e^{(B)}$, that
$$
A_1-A_2=6b_{12}\otimes b_{21},
\quad
A=\frac{1}{3}A_1+\frac{2}{3}A_2,
$$
$$
B_1-B_3= -6 b_{31} \otimes b_{13},
\quad
B=\frac{1}{3}B_1+\frac{2}{3}B_3,
$$
and that
$$
A-B=4b_{23}\otimes b_{32}.
$$

Motivated by these rank-one conditions, we recall the basis from \eqref{eq:basis} which is given by $\mathcal{R}:=\{n,b_{32},d\}$ and where $n,d\in\mathbb{S}^2$ are defined as in \eqref{eq:basis}.
We highlight once more that $d$ is orthogonal to all the chosen directions of lamination, i.e., to $b_{21},b_{13},b_{32}\in \text{span}(d)^\perp$. In what follows, we will use the $n, b_{32}, d$ components as coordinates for our construction, i.e., for notational simplicity, for every $x\in\R^3$ we will use the shorthand notation 
\begin{align}
\label{eq:coordinates}
z_1:=x\cdot n,
\quad
z_2:=x\cdot b_{32}
\quad\mbox{and}\quad
z_3:=x\cdot d.
\end{align}

\subsection{First level of lamination}
\label{sec:first}

We define an auxiliary first order branched laminate $\uaux:\R^3\to\R^3$ which is constant along the lines parallel to $d$, thus essentially two-dimensional.

Let $r>0$ be a sufficiently small parameter to be determined such that $\frac{1}{r}$ is an integer and let $\theta\in(\frac{1}{4},\frac{1}{2})$ be a fixed geometric constant.
We define a domain subdivision which corresponds to a branched laminate which is refining in direction $n$, oscillating in direction $b_{32}$ and constant in direction $d$.
Similar constructions are well-known in the literature (see, e.g., \cite{KM1,CC15,RT21}).
We repeat the argument for the convenience of the reader.

\smallskip

\textbf{Domain subdivision:}
for every $j\ge 0$ we define the parameters
\begin{equation}\label{eq:parameters}
L_j:=\frac{r}{2^j}, \quad Y_j:=\frac{1-\theta^j}{2}, \quad H_j:=\frac{\theta^j(1-\theta)}{2}.
\end{equation}
Here $j$ keeps track of the generation of the cell of the self-similar construction.
We define the reference cell of each generation, the others will be obtained via translation in direction $b_{32}$.

We recall the coordinate convention from \eqref{eq:coordinates}. Furthermore, for the sake of clarity of exposition, we only work in the half-space $\{z_1 \ge 0\}$, but all the following arguments can be reworked in $\{z_1<0\}$ in an analogous way. With these preliminary remarks, we define our domain decomposition of the unit cell of our construction at level $j\geq 0$ as follows:
\begin{align*}\label{eq:ref-cell}
\Omega^{(j)}_1 &:= \Big\{x\in\Omega : Y_j\le z_1\le Y_{j+1},\, 0\le z_2\le \frac{L_j}{4}\Big\}, \\
\Omega^{(j)}_2 &:= \Big\{x\in\Omega : Y_j\le z_1\le Y_{j+1},\, \frac{L_j}{4}\le z_2 \le \frac{L_j}{4}+\frac{L_j(z_1-Y_j)}{4H_j}\Big\}, \\
\Omega^{(j)}_3 &:= \Big\{x\in\Omega : Y_j \le z_1 \le Y_{j+1},\, \frac{L_j}{4}+\frac{L_j(z_1-Y_j)}{4H_j} \le z_2 \le \frac{L_j}{2}+\frac{L_j(z_1-Y_j)}{4H_j}\Big\} ,\\
\Omega^{(j)}_4 &:= \Big\{x\in\Omega : Y_j \le z_1 \le Y_{j+1},\, \frac{L_j}{2}+\frac{L_j(z_1-Y_j)}{4H_j} \le z_2 \le L_j\Big\},
\end{align*}
see Figure \ref{fig:cell1}.
\begin{figure}
\begin{center}
\includegraphics{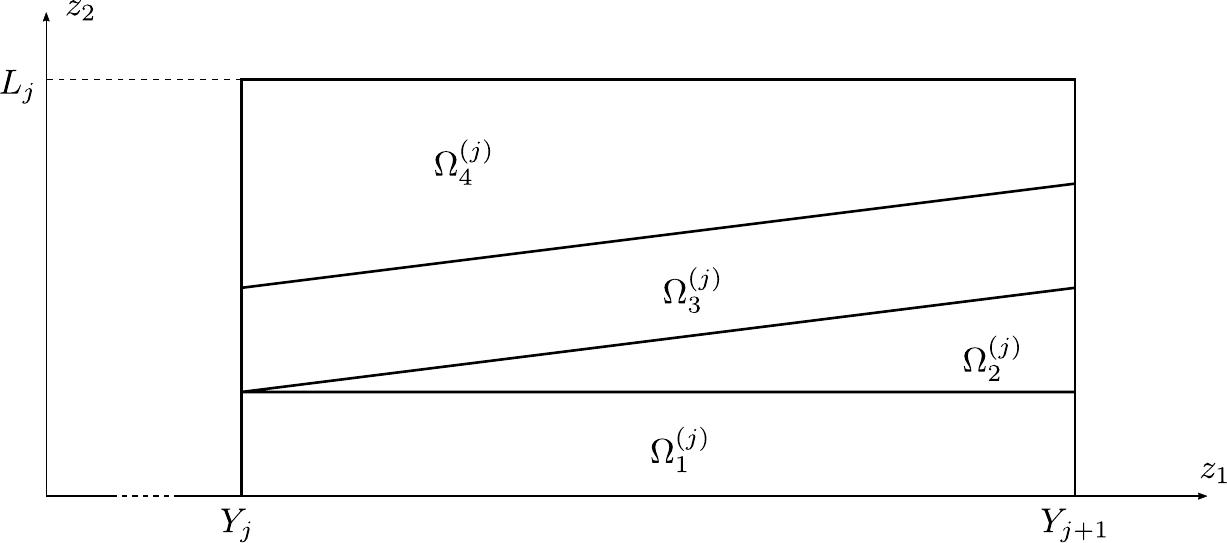}
\caption{The slice orthogonal to $d$ of the sets $\Omega_i^{(j)}$.}
\label{fig:cell1}
\end{center}
\end{figure}
Moreover, in our construction, we will stop the self-similar refinement as soon as $L_j\ge H_j$. As a consequence, the last branching generation is indexed by $j_0$ which is the largest integer such that $L_{j_0}<H_{j_0}$.
We highlight that by choosing $r<\frac{1-\theta}{2}$ for any $\theta<\frac{1}{2}$, the index $j_0$ is well defined and strictly positive.

We also define
\begin{align*}
\Omega^{(j_0+1)}_1 &:= \Big\{x\in\Omega : Y_{j_0+1}\le z_1\le \frac{1}{2},\, 0\le z_2\le \frac{L_{j_0+1}}{2}\Big\}, \\
\Omega^{(j_0+1)}_2 &:= \Big\{x\in\Omega : Y_{j_0+1}\le z_1\le \frac{1}{2},\, \frac{L_{j_0+1}}{2}\le z_2\le L_{j_0+1}\Big\},
\end{align*}
which is a covering of the $z_1$-neighbourhood of $\p_0\Omega$. Here we will exploit a cut-off argument to match the zero boundary condition.
We also write $\Omega^{(j)}:=\Omega^{(j)}_1\cup \Omega^{(j)}_2\cup \Omega^{(j)}_3\cup \Omega^{(j)}_4$ for every $j\in\{0,\dots,j_0\}$ and $\Omega^{(j_0+1)}:=\Omega^{(j_0+1)}_1\cup \Omega^{(j_0+1)}_2$.

In order to cover the whole domain, for each generation $j$ we have $\frac{2^j}{r}$ many identical copies of $\Omega^{(j)}$.
More precisely, we have
\begin{equation}\label{eq:covering}
\{x\in\Omega : z_1 \ge 0\}=\text{int}\Big(\bigcup_{j=0}^{j_0+1} \bigcup_{k=-\frac{2^{j-1}}{r}}^{\frac{2^{j-1}}{r}-1}\big(\Omega^{(j)}+kL_j b_{32}\big)\Big).
\end{equation}

\smallskip

Given the covering above, we can infer the following first order branching construction (analogous to \cite[Lemma 3.1]{RT21}).

\begin{lem}\label{lem:1tree}
Let $A,B\in\mathbb{R}^{3\times 3}$ be as above.
Then there exist $\uaux\in C^{0,1}(\R^3;\R^3)$ and $\chiaux\in BV_{\rm loc}(\R^3;\{e^{(A)},e^{(B)}\})$ complying with
\begin{equation}\label{eq:bc-aux}
\p_d\uaux \equiv 0 \text{ a.e.\ in $\R^3$},
\quad
\uaux(x)=0 \text{ if } x\in\p_0\Omega,
\end{equation}
such that
\begin{equation}\label{eq:reg-aux}
\nabla\uaux\in BV_{\rm loc}(\R^3;\R^{3\times 3}),
\quad
\|\nabla\uaux\|_{L^\infty(\R^3)}\lesssim 1,
\end{equation}
and that
\begin{equation}\label{eq:energy-aux}
\int_\Omega |e(\nabla\uaux(x))-\chiaux(x)|^2 dx+\epsilon|D\chiaux|(\Omega)\lesssim r^2+\epsilon\frac{1}{r}.
\end{equation}
\end{lem}
\begin{proof}
We divide the proof into four steps.

\smallskip

\emph{Step 1: Reference cell construction.}
For every $j\in\{0,\dots,j_0\}$ we define $u^{(j)}:\Omega^{(j)}\to\R^3$ as
$$
u^{(j)}(x):=
\begin{cases}
Ax & x\in \Omega^{(j)}_1, \\
Bx+L_j b_{23} 
& x\in \Omega^{(j)}_2, \\
A^{(j)} x+\frac{L_j}{H_j}Y_j b_{23} & x\in\Omega^{(j)}_3, \\
Bx + 2L_j b_{23} 
& x\in\Omega^{(j)}_4,
\end{cases}
$$
with $A^{(j)} := A -\frac{L_j}{H_j} b_{23} \otimes n$.
By construction $u^{(j)}$ is Lipschitz continuous, $\nabla u^{(j)}\in\{A,B,A^{(j)}\}$ is $BV$ and $\|\nabla u^{(j)}\|_{L^\infty(\Omega^{(j)})} \lesssim 1$.
Since $A d=B d=A^{(j)} d=0$, we also have $\p_d u^{(j)}\equiv0$ a.e..
Moreover, it holds that
\begin{equation}\label{eq:bc-refcell1}
u^{(j)}(x)=0,
\quad \text{if }x\in\Omega^{(j)} \mbox{ such that } z_2\in\{0,L_j\},
\end{equation}
and
\begin{equation}\label{eq:bc-refcell2}
u^{(j)}(x) = f_j(z_2) b_{23},
\quad \text{if } x\in \Omega^{(j)} \mbox{ such that } z_1=Y_j,
\end{equation}
\begin{equation}\label{eq:bc-refcell3}
u^{(j)}(x) =
\begin{cases}
\frac{1}{2}f_j(2z_2) b_{23} & 0\le z_2\le\frac{L_j}{2}, \\
\frac{1}{2}f_j(2z_2-L_j) b_{23} & \frac{L_j}{2}\le z_2\le L_j,
\end{cases}
\quad x\in \Omega^{(j)} \mbox{ and such that } z_1=Y_{j+1}.
\end{equation}
Here we have defined $f_j(t)=\min\{2t,2L_j-2t\}$.

We eventually define $u^{(j_0+1)}:\Omega^{(j_0+1)}\to\R^3$ as
$$
u^{(j_0+1)}(x)=\frac{1-2z_1}{1-2Y_{j_0+1}} f_{j_0+1}(z_2)b_{23}, 
$$
which is a cut-off to attain zero boundary conditions on the whole lateral boundary $\p_0\Omega$.
The function $u^{(j_0+1)}$ complies with \eqref{eq:bc-refcell1}, \eqref{eq:bc-refcell2} and
\begin{equation}\label{eq:bc-refcell-cutoff}
u^{(j_0+1)}(x)=0 \quad \text{if } x\in\Omega^{(j_0+1)} \mbox{ is such that } z_1=\frac{1}{2}.
\end{equation}
Moreover, \eqref{eq:parameters} and the definition of $j_0$ yield $\|\nabla u^{(j_0+1)}\|_{L^\infty(\Omega^{(j_0+1)})}\lesssim 1$, and it satisfies $\p_d u^{(j_0+1)}\equiv0$ almost everywhere. 

\smallskip

\emph{Step 2: Compatibility conditions.}
We show that the reference constructions defined in Step 1 can be concatenated continuously.

Indeed, by \eqref{eq:bc-refcell1} for every $j\in\{0,\dots,j_0+1\}$ each $u^{(j)}$ can be extended periodically (in the $z_2$-direction) to a continuous function (not relabelled) $u^{(j)}:\{x\in\Omega: Y_j\le z_1\le Y_{j+1}\}\to\R^3$. 
Moreover, for every $j\in\{0,\dots,j_0\}$, from \eqref{eq:parameters}, \eqref{eq:bc-refcell2} and \eqref{eq:bc-refcell3} we obtain that
$$
u^{(j)}(x)=u^{(j+1)}(x),
\quad \text{if } x\in\Omega \mbox{ with } z_1=Y_{j+1}.
$$

\smallskip

\emph{Step 3: Full construction.}
With the previous preparation we are now in a position to define $\uaux:\Omega\to\R^3$ by simply setting
$$
\uaux(x):= u^{(j)}(x), \quad \text{if } x\in\Omega \mbox{ with } Y_j\le z_1\le Y_{j+1} , \ j \in \{0,\dots,j_0+1\}.
$$
On $\Omega\cap\{z_1<0\}$ we define it by reflection.
By construction the properties in \eqref{eq:bc-aux} and \eqref{eq:reg-aux} are satisfied.
Accordingly, we define the phase indicator as
$$
\chiaux(x)=
\begin{cases}
e^{(A)} & \text{on } \Omega^{(j)}_i+kL_j b_{32} \text{ for } i=\{1,3\},\, j\in\{0,\dots,j_0+1\},\, k\in\{-\frac{2^{j-1}}{r},\dots,\frac{2^{j-1}}{r}-1\},  \\
e^{(B)} & \text{on } \Omega^{(j)}_i+kL_j b_{32} \text{ for } i=\{2,4\},\, j\in\{0,\dots,j_0+1\},\, k\in\{-\frac{2^{j-1}}{r},\dots,\frac{2^{j-1}}{r}-1\}. \\
\end{cases}
$$
As above, on $\Omega\cap\{z_1<0\}$, we define it by reflection.
By \eqref{eq:bc-aux} $u_{aux}$ is extended to be constant in the $d$-direction and to vanish outside $[-\frac{1}{2},\frac{1}{2}]n\times[-\frac{1}{2},\frac{1}{2}]b_{32}\times\R d$.

\emph{Step 4: Energy contribution.}
By \eqref{eq:parameters}, \eqref{eq:covering} and by construction, the surface energy is controlled by
\begin{equation}\label{eq:surf-en-ord1}
|D\chiaux|(\Omega) \lesssim \sum_{j=0}^{j_0+1} \frac{2^j}{r}\mathcal{H}^2(\p\Omega_1^{(j)}) \lesssim \sum_{j\ge0} \frac{2^j}{r}H_j \lesssim \sum_{j\ge0} \frac{(2\theta)^j}{r}\lesssim \frac{1}{r}.
\end{equation}
The only nonzero elastic energy contributions (apart from the cut-off region) are on $\Omega_3^{(j)}+kL_j b_{32}$, which are identical to the contribution on the reference cell $\Omega_3^{(j)}$.
In there
$$
\int_{\Omega_3^{(j)}}|e(\nabla\uaux(x))-e^{(A)}|^2dx = \int_{\Omega_3^{(j)}}|e(A^{(j)})-e^{(A)}|^2dx \sim \frac{L_j^2}{H_j^2}L_j H_j \sim \frac{r^3}{(8\theta)^j}.
$$
In the cut-off region, by the boundedness of the gradient and the definition of $j_0$ we have
$$
\int_{\Omega^{(j_0+1)}} |e(\nabla\uaux(x))-\chiaux(x)|^2 dx \lesssim L_{j_0+1}H_{j_0+1} \lesssim L_{j_0+1}H_{j_0+1}\Big(\frac{L_{j_0+1}}{H_{j_0+1}}\Big)^2 \lesssim \frac{r^3}{(8\theta)^{j_0+1}}.
$$
Thus, summing over $j$, since we have $\frac{2^j}{r}$ identical copies of $\Omega^{(j)}$ and since $4\theta>1$, we infer
\begin{equation}\label{eq:el-en-ord1}
\int_{\Omega}|e(\nabla\uaux(x))-\chiaux(x)|^2dx \lesssim \sum_{j=0}^{j_0+1}\frac{2^j}{r}\int_{\Omega^{(j)}}\Big|e(A^{(j)})-e(A) \Big|^2dx \lesssim \sum_{j\ge0}\frac{r^2}{(4\theta)^j} \lesssim r^2.
\end{equation}
Collecting \eqref{eq:surf-en-ord1} and \eqref{eq:el-en-ord1} yields the claimed result of the lemma.
\end{proof}

\begin{rmk}\label{rmk:shear}
We highlight that $\Omega_3^{(j)}=\phi_j(\Omega_1^{(j)})$ where $\phi_j$ is the (shifted) shear 
$$
\phi_j(x)=M_j x+\frac{L_j}{4}\Big(1-\frac{Y_j}{H_j}\Big) b_{32}
$$
and $M_j=Id+\frac{L_j}{4H_j} b_{32}\otimes n$.
Notice also that $A^{(j)}=(A-B) M_j^{-1}+B$. 
For later use (c.f. the definition of the displacement in Section \ref{sec:A3}), we observe that the displacement $u^{(j)}|_{\Omega_3^{(j)}}$ can be obtained from the displacement $u^{(j)}|_{\Omega_1^{(j)}}$ by using these components: Indeed, first ``isolating the oscillating part'', i.e., writing $\tilde u^{(j)}(x):=u^{(j)}(x)-Bx$ for every $x\in\Omega_1^{(j)}$, we set
$$
\tilde u^{(j)}|_{\Omega_3^{(j)}}(x'):= \tilde u^{(j)}|_{\Omega_1^{(j)}}(\phi_j^{-1}(x')), \quad x'\in\Omega_3^{(j)}.
$$
With this definition, a short computation then implies that
\begin{equation}
\label{eq:change-var}
u^{(j)}|_{\Omega_3^{(j)}}(x') := \tilde u^{(j)}|_{\Omega_3^{(j)}}(x')+(Bx'+L_j b_{23}).
\end{equation}
We note that $\uaux$ is constructed from $u^{(j)}$ in the way just outlined. Hence, if now a function $v^{(j)}:\Omega^{(j)} \rightarrow \R^3$ is obtained through an analogous shear procedure and if $v^{(j)}(x) = \uaux(x)$ for $x\in \partial_0 \Omega^{(j)}_1$, then by construction $v^{(j)}$ also coincides with $\uaux(x)$ for $x\in\p_0\Omega_3^{(j)}$. We will make use of this in Section \ref{sec:A3} below.
\end{rmk}

\subsection{Second level of branching}
\label{sec:second}

We now turn to the second order of lamination.

In the subdomains $\Omega^{(j)}_1$ and $\Omega^{(j)}_3$ we define a branched laminate with gradients $A_1$ and $A_2$, attaining the affine boundary condition $\uaux$.
The construction inside $\Omega^{(j)}_3$ will be obtained from the one in $\Omega_1^{(j)}$ with a shear argument.

In the subdomains $\Omega^{(j)}_2$ we define a branched laminate with gradients $B_1$ and $B_3$, attaining the affine boundary condition $\uaux$.
This will require another self-similar refinement due to the triangular shape of the sections (orthogonal to $d$) of $\Omega^{(j)}_2$.
Exploiting the  two constructions in $\Omega^{(j)}_1$, $\Omega_2^{(j)}$, we subsequently obtain the branched laminate inside $\Omega^{(j)}_4$ easily. We split the discussion of these constructions into several subsections and first deal with the $A_1, A_2$ lamination in the domains $\Omega^{(j)}_1$ in Section \ref{sec:2-ord-lam-1}, in the domains $\Omega^{(j)}_3$ in Section \ref{sec:A3} and the lamination between $B_1, B_2$ in $\Omega_j^{(2)}$ in Section \ref{sec:B2}. In Section \ref{sec:final} we then combine all previous estimates into the desired upper scaling bound.

\subsubsection{Lamination of $A_1$ and $A_2$ inside $\Omega_1^{(j)}$}
\label{sec:2-ord-lam-1}

Let $r_2>0$ be a parameter to be fixed later (in the proof of Theorem \ref{thm:upper} below) with ratio $\frac{r_2}{r}$ being sufficiently small.
We define a branched laminate, refining in $b_{32}$, oscillating in $b_{21}$ and being constant in $d$ direction. The non-orthogonality of the second order of lamination with respect to the first order will provide new difficulties in this construction compared to the constructions from \cite{RT21}.

In what follows, we consider the following domain subdivision:
We cut $\Omega_1^{(j)}$ into slices orthogonal to $b_{21}$ of amplitude $\frac{r_2}{2^j}$, starting from the bottom-left corner and ending on the top-right one, assuming for simplicity that $\frac{2^jH_j}{r_2}$ is integer.
This assumption is not restrictive, as we may always realize this by changing the value of $r_2$ for each $j$ suitably. 
For further insights on the choice of the length scale $\frac{r_2}{2^j}$ we point the interested reader to the discussion in \cite[Remark 5.1]{RT21}.

Due to the non-orthogonality of $b_{21}$ and $n$, this subdivision leaves the corners of $\Omega_1^{(j)}$ uncovered.
Specifically,
$$
\Omega_1^{(j)}=T_l\cup T_r\cup S,
$$
where
\begin{align*}
S & := \Big\{x\in\Omega_1^{(j)} : \frac{1}{\sqrt{3}}z_2+Y_j\le z_1\le\frac{1}{\sqrt{3}}z_2+Y_{j+1}-\frac{1}{4\sqrt{3}}L_j \Big\}, \\
T_l & := \Big\{x\in\Omega_1^{(j)} : Y_j\le z_1\le\frac{1}{\sqrt{3}}z_2+Y_j\Big\}, \\
T_r & := \Big\{x\in\Omega_1^{(j)} : \frac{1}{\sqrt{3}}z_2+Y_{j+1} - \frac{1}{4\sqrt{3}}L_j\le z_1\le Y_{j+1}\Big\},
\end{align*}
see Figure \ref{fig:oscA1}. For ease of notation, we have here dropped the index $j$ in the definition of the sets $T_l, T_r, S$.
\begin{figure}
\begin{center}
\includegraphics{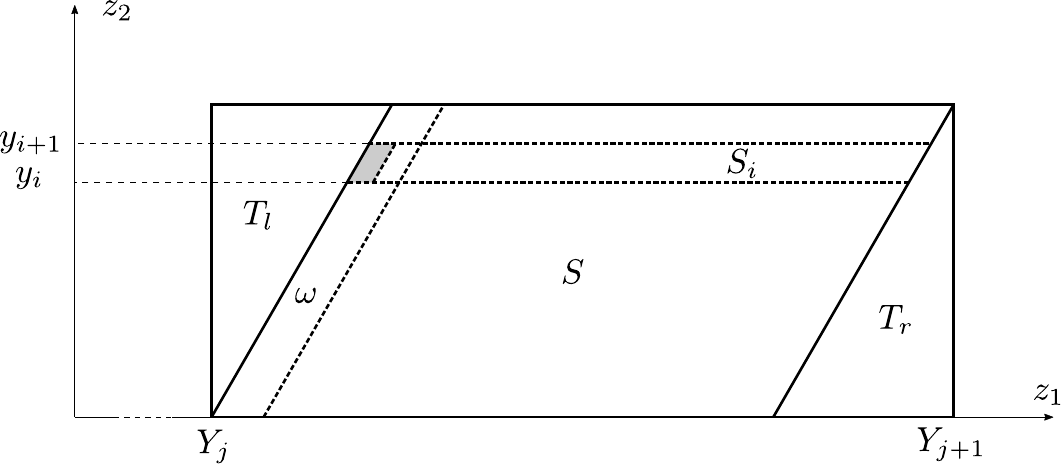}
\caption{The domain subdivision of $\Omega_1^{(j)}$ into the sets $S$, $T_l$ and $T_r$.
The regions $S_i$ and $\omega$ (cf. Steps 1 and 3 in the proof of Lemma \ref{lem:2ord1}) are also depicted.
The shaded region represents the cell $\omega^{(i)}$ which is depicted in more detail in Figure \ref{fig:oscA2}.}
\label{fig:oscA1}
\end{center}
\end{figure}
Here we have used that $b_{21}= \frac{\sqrt{3}}{2}n- \frac{1}{2}b_{32}$.
Thus, in what follows, we provide a ``standard'' covering for $S$ and treat the corners $T_l$ and $T_r$ separately (cf. Step 2 of the proof of Lemma \ref{lem:2ord1} and Figure \ref{fig:corner}).

\textbf{Second order domain subdivision:}
Let $i\in\N$ denote the generation of the self-similar refinement of this second order construction.
We define the new parameters
\begin{equation}\label{eq:parameters2}
\ell_i := \frac{r_2}{2^j2^i},
\quad
y_i := \frac{L_j}{8}(2-\theta^i),
\quad
h_i := \frac{L_j}{8}(1-\theta)\theta^i,
\end{equation}
and the associated reference cells 
\begin{align*}
\omega_1^{(i)} &:= \Big\{x\in\Omega_1^{(j)} : y_i\le z_2\le y_{i+1},\, \frac{1}{\sqrt{3}}z_2\le z_1-Y_j\le\frac{1}{\sqrt{3}}z_2+\frac{\ell_i}{6}\Big\}, \\
\omega_2^{(i)} &:= \Big\{x\in\Omega_1^{(j)} : y_i\le z_2\le y_{i+1},\, \frac{1}{\sqrt{3}}z_2+\frac{\ell_i}{6}\le z_1-Y_j\le \sigma_i(z_2-y_i)+\frac{y_i}{\sqrt{3}}+\frac{\ell_i}{6}\Big\}, \\
\omega_3^{(i)} &:= \Big\{x\in\Omega_1^{(j)} : y_i\le z_2\le y_{i+1},\, \sigma_i(z_2-y_i)+\frac{\ell_i}{6}+\frac{y_i}{\sqrt{3}}\le z_1-Y_j\le\sigma_i(z_2-y_i)+\frac{\ell_i}{3}+\frac{y_i}{\sqrt{3}}\Big\}, \\
\omega_4^{(i)} &:= \Big\{x\in\Omega_1^{(j)} : y_i\le z_2\le y_{i+1},\, \sigma_i(z_2-y_i)+\frac{\ell_i}{3}+\frac{y_i}{\sqrt{3}}\le z_1-Y_j\le\frac{1}{\sqrt{3}}z_2+\ell_i\Big\},
\end{align*}
where $\sigma_i:=(\frac{\ell_i}{3 h_i}+\frac{1}{\sqrt{3}})$, see Figure \ref{fig:oscA2}.
Again the subdivision stops at $i_0$, the largest index such that $\ell_{i_0}\leq h_{i_0}$.
We define the boundary layer cells as
\begin{align*}
\omega_1^{(i_0+1)} &:= \Big\{x\in\Omega_1^{(j)} : y_{i_0+1}\le z_2\le\frac{L_j}{4},\, \frac{1}{\sqrt{3}}z_2\le z_1-Y_j\le \frac{1}{\sqrt{3}}z_2+\frac{\ell_{i_0+1}}{3}\Big\}, \\
\omega_2^{(i_0+1)} &:= \Big\{x\in\Omega_1^{(j)} : y_{i_0+1}\le z_2\le\frac{L_j}{4},\, \frac{1}{\sqrt{3}}z_2+\frac{\ell_{i_0+1}}{3}\le z_1-Y_j\le\frac{1}{\sqrt{3}}z_2+\ell_{i_0+1}\Big\}.
\end{align*}
\begin{figure}
\begin{center}
\includegraphics{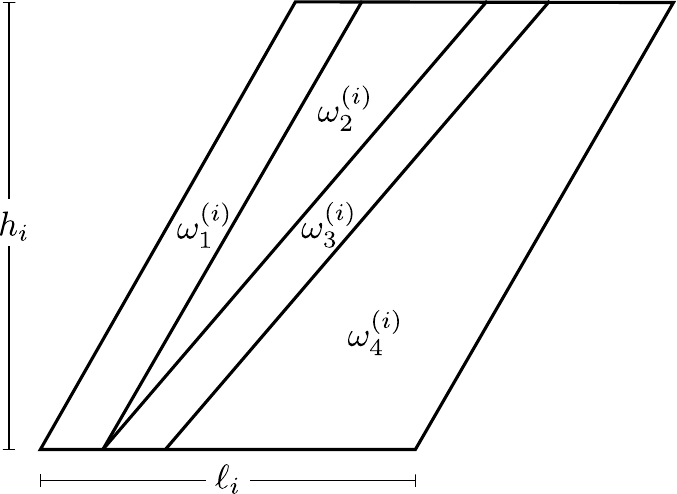}
\caption{The domain subdivision of $\omega^{(i)}$.}
\label{fig:oscA2}
\end{center}
\end{figure}\\
We also write $\omega^{(i)}=\omega^{(i)}_1\cup\omega^{(i)}_2\cup\omega^{(i)}_3\cup\omega^{(i)}_4$.
For each generation $i$ we have $\frac{2^{i+j} H_j}{r_2}$ many identical copies of $\omega^{(i)}$, i.e.,
\begin{equation}\label{eq:covering2}
\Big\{x\in S : z_2 \ge \frac{L_j}{8}\Big\}=\bigcup_{i=0}^{i_0+1} \bigcup_{k=0}^{\frac{2^{i+j}H_j}{r_2}-1}\big(\omega^{(i)}+k\ell_i n\big).
\end{equation}
Having fixed this notation, we deduce the following upper bound construction in the cells $\Omega^{(j)}_1$.
For the sake of clarity of exposition, in the sequel we will use the notation
$$
E_\epsilon(v,\psi;\Omega'):=\int_{\Omega'}|v(x)-\psi(x)|^2dx+\epsilon|D\psi|(\text{int}(F)),
$$
for any $\Omega'\subset\Omega$ is a Lipschitz set, $v\in H^1(\text{int}(\Omega');\R^3)$, $\psi\in BV(\text{int}(\Omega');K)$.

\begin{lem}\label{lem:2ord1}
Let $A_1$ and $A_2$ be as above.
Then there exist $u^{(j,1)}\in C^{0,1}(\Omega^{(j)}_1;\R^3)$ and $\chi^{(j,1)}\in BV({\rm int(}\Omega^{(j)}_1;\{e^{(1)},e^{(2)}\})$ complying with
\begin{equation}\label{eq:bc-2ord1}
\p_d u^{(j,1)}\equiv 0 \text{ a.e. in $\Omega_1^{(j)}$},
\quad
u^{(j,1)}(x)=\uaux(x) \text{ if } x\in\p_0\Omega^{(j)}_1,
\end{equation}
such that
\begin{equation}\label{eq:reg-2ord1}
\nabla u^{(j,1)}\in BV({\rm int(}\Omega^{(j)}_1;\R^{3\times 3}),
\quad
\|\nabla u^{(j,1)}\|_{L^\infty(\Omega^{(j)}_1)}\lesssim 1,
\end{equation}
and satisfying
\begin{equation}\label{eq:2ord1-en}
E_\epsilon(u^{(j,1)},\chi^{(j,1)};\Omega^{(j)}_1) \lesssim \Big(\frac{\theta}{2}\Big)^j\frac{r_2^2}{r}+\frac{r_2 r}{4^j}+\epsilon \theta^j\frac{r}{r_2}.
\end{equation}
\end{lem}

Apart from the non-orthogonality (which gives rise to the new domains $T_l, T_r$), the construction is analogous as the one presented in the proof of Lemma \ref{lem:1tree}. Hence, we just highlight the main differences in the proof.

\begin{proof}
We divide the proof into four steps.

\emph{Step 1: Second order branching construction inside $S$.}
For every $i\in\{0,\dots,i_0\}$ we define $v^{(i)}\in W^{1,\infty}(\omega^{(i)};\R^3)$ as
$$
v^{(i)}(x) =
\begin{cases}
A_1(x-Y_jn) & x\in\omega_1^{(i)}, \\
A_2(x-Y_jn)+\frac{\sqrt{3}}{2}\ell_i b_{12} & x\in\omega_2^{(i)}, \\
A_1^{(i)}(x-Y_jn)+\sqrt{3}\frac{\ell_i}{h_i}y_ib_{12} & x\in\omega_3^{(i)}, \\
A_2(x-Y_jn)+\sqrt{3}\ell_ib_{12} & x\in\omega_4^{(i)},
\end{cases}
$$
where $A_1^{(i)}=A_1-\sqrt{3}\frac{\ell_i}{h_i}b_{12}\otimes b_{32}$.
For $i=i_0+1$ we use a cut-off argument (to attain $Ax$) analogous to that of the proof of Lemma \ref{lem:1tree}, denoting the corresponding displacement as $v^{(i_0+1)}$.
With a slight abuse of notation, we define $v^{(i)}$ on the whole stripe
$$
S_i:=\{x\in S : y_i\le z_2\le y_{i+1}\}
$$
(cf. Figure \ref{fig:oscA1}) with an affine extension; namely, $v^{(i)}(x)=w^{(i)}(x)+Ax$ where $w^{(i)}$ is the $\omega^{(i)}$-periodic extension of $v^{(i)}(x)-Ax$ on $S_i$.
Correspondingly we define
$$
\psi^{(i)}(x)=
\begin{cases}
e^{(1)} & x\in\omega_1^{(i)}\cup\omega_3^{(i)}, \\
e^{(2)} & x\in\omega_2^{(i)}\cup\omega_4^{(i)},
\end{cases}
$$
which we extend $\omega^{(i)}$-periodically on $S_i$.
Arguing analogously as in Step 1 and 2 of the proof of Lemma \ref{lem:1tree} (cf. \eqref{eq:bc-refcell2} and \eqref{eq:bc-refcell3}), by construction we have that $v^{(i)}(x)=v^{(i+1)}(x)$ for $x\in S$ with $z_2=y_{i+1}$.
Thus, 
$$
v(x):=v^{(i)}(x)
\quad \text{if } x\in S_i,
$$
defines a Lipschitz function on $S$.
We also define $\psi(x):=\psi^{(i)}(x)$ as $x\in S_i$.
Hence, $v\in C^{0,1}(S;\R^3)$ with affine boundary data $\uaux(x)=A x$ and $\psi\in BV(S;\{e^{(1)},e^{(2)}\})$.

\emph{Step 2: Filling the corners $T_l$ and $T_r$.}
We cut $T_l$ into slices orthogonal to $b_{21}$.
For $m\ge 0$, we iteratively define the slices
\begin{align*}
\omega_m &:= \Big\{x\in T_l : \sqrt{3}\sum_{m'=0}^{m}\rho_{m'}\le z_2\le\frac{L_j}{4},\, \frac{1}{\sqrt{3}}z_2-\rho_m\le z_1-Y_j+\sum_{m'=0}^{m-1}\rho_{m'}\le\frac{1}{\sqrt{3}}z_2\Big\},
\end{align*}
where we set
$$
\rho_m:=\frac{r_2}{2^j}\Big(1-\frac{4\sqrt{3} r_2}{2^j L_j}\Big)^m
\quad \text{and} \quad w_m:= \frac{L_j}{4}\Big(1-\frac{4\sqrt{3}r_2}{2^j L_j}\Big)^{m+1},
$$
see Figure \ref{fig:corner}.
\begin{figure}
\begin{center}
\includegraphics{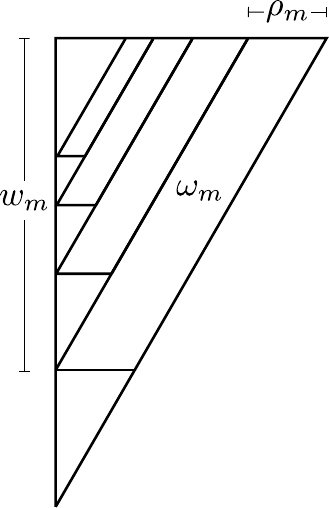}
\caption{The domain subdivision of $T_l$ into sets $\omega_m$.}
\label{fig:corner}
\end{center}
\end{figure}
We stop as $m\le m_0$, with $m_0$ being the largest index such that $\sum_{m=0}^{m_0}\rho_m<\frac{L_j}{4\sqrt{3}}-\frac{r_2}{2^j}$.
Notice that with this choice it holds
\begin{equation}\label{eq:corner-area}
\Big|T_l\setminus\Big(\bigcup_{m=0}^{m_0}\omega_{m}\Big)\Big| \lesssim \sum_{m\ge0}\rho_m^2+\Big(\frac{r_2}{2^j}\Big)^2 \lesssim \frac{r_2}{2^j}L_j+\frac{r_2^2}{4^{j}}.
\end{equation}
In each $\omega_m$ we define $v_l^{(m)}$ and $\psi_l^{(m)}$ analogously as $v^{(i)}$ and $\psi^{(i)}$ in Step 1.
Eventually, for every $x\in T_l$ we set
$$
v_l(x):=
\begin{cases}
v_l^{(m)}(x) & x\in\omega_m, \ m\in\{0,\dots,m_0\}, \\
\uaux(x) & \text{otherwise},
\end{cases}
\quad
\psi_l(x):=
\begin{cases}
\psi_l^{(m)}(x) & x\in\omega_m, \ m\in\{0,\dots,m_0\}, \\
e^{(2)} & \text{otherwise.}
\end{cases}
$$
A completely analogous argument on $T_r$ is used to define $v_r$ and $\psi_r$.

\emph{Step 3: Energy contribution of a single ``tree'' of branching in $S$.}
We consider the reference cell
$$
\omega := \Big\{x\in\Omega : 0\le z_2\le\frac{L_j}{4},\,
\frac{1}{\sqrt{3}}z_2\le z_1-Y_j\le\frac{r_2}{2^j}+\frac{1}{\sqrt{3}}z_2\Big\}\subset S,
$$
which is the first slice of $S$ (Figure \ref{fig:oscA1}).
The restriction of $v$ and $\psi$ on $\omega$ corresponds to a single ``tree'' of our branching construction.

Again the energy contributions originates only from $\omega_3^{(i)}$ and it holds
$$
\int_{\omega_3^{(i)}}|e(A_1^{(i)})-e^{(1)}|^2 dx \lesssim \Big(\frac{\ell_{i}}{h_{i}}\Big)^2\ell_{i}h_{i} \lesssim \frac{r_2^3}{L_j(8\theta)^i8^j}.
$$
In $\omega$ we have $2^i$ of these contributions for each generation $i$. Thus,
$$
\int_\omega |e(\nabla v)-\psi|^2 dx \lesssim \sum_{i\ge 0} \frac{r_2^3}{L_j(4\theta)^{i}8^j} \lesssim \frac{r_2^3}{8^jL_j},
$$
where we also include the cut-off contributions working as in the proof of Lemma \ref{lem:1tree}.
The surface energy contributions originating from the second order branched laminates in a cell of generation $i$ within $\omega$ in turn simply corresponds to $\sim 2^{i}h_{i}$. Hence,
$$
|D\psi|({\rm int(}\omega) \lesssim \sum_{i\ge0} L_j (2\theta)^{i} \lesssim L_j.
$$
The total energy contribution on $\omega$ is
\begin{equation}\label{eq:en-1tree-2ord}
E_\epsilon(v,\psi;\omega) \lesssim \frac{r_2^3}{8^jL_j} + \epsilon L_j.
\end{equation}

\emph{Step 4: Total energy contribution inside $\Omega_1^{(j)}$.}
We finally define the construction on the whole $\Omega_1^{(j)}$ as follows
$$
u^{(j,1)}(x):=
\begin{cases}
v(x) & x\in S, \\
v_l(x) & x\in T_l, \\
v_r(x) & x\in T_r,
\end{cases}
\quad\text{and}\quad
\chi^{(j,1)}(x):=
\begin{cases}
\psi(x) & x\in S, \\
\psi_l(x) & x\in T_l, \\
\psi_r(x) & x\in T_r.
\end{cases}
$$
The contribution in $\Omega_1^{(j)}$ splits as
$$
E_\epsilon(u^{(j,1)},\chi^{(j,1)};\Omega_1^{(j)}) \le E_\epsilon(v,\psi;S)+E_\epsilon(v_l,\psi_l;T_l)+E_\epsilon(v_r,\psi_r;T_r).
$$
Since there are $\frac{2^jH_j}{r_2}$ many identical copies of $\omega$ inside $S$ we infer from \eqref{eq:en-1tree-2ord} that
\begin{equation}\label{eq:en-S}
E_\epsilon(v,\psi;S) \lesssim \frac{2^jH_j}{r_2}E_\epsilon(v,\psi;\omega)\lesssim \frac{\theta^j r_2^2}{2^jr}+\epsilon\frac{\theta^j r}{r_2}\lesssim \Big(\frac{\theta}{2}\Big)^j\frac{r_2^2}{r}+\epsilon \theta^j\frac{r}{r_2}.
\end{equation}
By construction of the domains $\omega_m$ (as in Step 2) 
the definitions of $(v,\psi)$ and $(v_l^{(m)},\psi_l^{(m)})$ are completely analogous.
Hence, the energy in $\omega_m$ can be obtained by replacing $\frac{r_2}{2^j}$ and $\frac{L_j}{4}$ with $\rho_m$ and $w_m $ in \eqref{eq:en-1tree-2ord}, respectively.
Thus, from \eqref{eq:en-1tree-2ord} and \eqref{eq:corner-area}, by summing over $m$ we get
\begin{equation}\label{eq:en-Tl}
E_\epsilon(v_l,\psi_l;T_l) \lesssim \Big(\sum_{m\ge0} \Big(\frac{\rho_m^3}{w_m}+\epsilon w_m \Big) \Big)+\frac{r_2}{2^j}L_j+\frac{r_2^2}{4^{j}} \lesssim \frac{r_2^2}{4^j}+\epsilon2^j\frac{L_j^2}{r_2}+\frac{r_2}{2^j}L_j.
\end{equation}
Recalling that for $j\le j_0$ we have that $L_j\lesssim\theta^j$, the surface energy term above can be controlled as $\epsilon 2^j\frac{L_j^2}{r_2}\lesssim\epsilon\theta^j\frac{ 2^jL_j}{r_2}=\epsilon\theta^j\frac{r}{r_2}$.
Summing \eqref{eq:en-S} and \eqref{eq:en-Tl}, therefore yields
$$
E_\epsilon(u^{(j,1)},\chi^{(j,1)};\Omega_1^{(j)}) \lesssim \Big(\frac{\theta}{2}\Big)^j\frac{r_2^2}{r}+\frac{r_2^2}{4^j}+\frac{r_2 r}{4^j}+\epsilon \theta^j\frac{r}{r_2},
$$
which results in the desired estimate.
\end{proof}

\subsubsection{Lamination of $A_1$ and $A_2$ inside $\Omega^{(j)}_3$}
\label{sec:A3}

In order to define the construction in this domain, we use the shear argument introduced in Remark \ref{rmk:shear}.
Let $u^{(j,1)}$ be given by Lemma \ref{lem:2ord1}, then we define
$$
u^{(j,3)}(x) := u^{(j,1)}(\phi_j^{-1}(x))+B(x-\phi_j^{-1}(x))+L_j b_{23}
\quad \text{and} \quad
\chi^{(j,3)} := \chi^{(j,1)}(\phi_j^{-1}(x)).
$$
Thus, $u^{(j,3)}\in C^{0,1}(\Omega^{(j)}_3;\R^3)$ with
$$
\nabla u^{(j,3)}\in BV(\rm int(\Omega^{(j)}_3;\R^{3\times 3}),
\quad
\|\nabla u^{(j,3)}\|_{L^\infty(\Omega^{(j)}_3}\lesssim 1,
$$
and by Remark \ref{rmk:shear} (since this holds in $\partial_0 \Omega_1^{(j)}$)
$$
u^{(j,3)}(x)=\uaux(x)
\text{ if } x\in\p_0\Omega^{(j)}_3.
$$
Moreover, letting $y=\phi_j^{-1}(x)$ and noticing that
$$
\nabla u^{(j,3)}(x)=\nabla u^{(j,1)}(y)M_j^{-1}+B(Id-M_j^{-1})=\nabla u^{(j,1)}(y)-\frac{L_j}{4H_j}(\nabla u^{(j,1)}(y)-B)b_{32}\otimes n
$$
by this change of variables we get
\begin{equation}\label{eq:en-second}
\begin{split}
\int_{\Omega^{(j)}_3}|e(\nabla u^{(j,3)}(x))-\chi^{(j,3)}(x)|^2 dx & \lesssim \int_{\Omega^{(j)}_1}|e(\nabla u^{(j,1)}(y))-\chi^{(j,1)}(y)|^2 dy +\Big(\frac{L_j}{H_j}\Big)^2|\Omega^{(j)}_1| \\
& \lesssim \Big(\frac{\theta}{2}\Big)^j\frac{r_2^2}{r}+\frac{r_2 r}{4^j}+\frac{r^3}{(8\theta)^j}.
\end{split}
\end{equation}

\subsubsection{Lamination of $B_1$ and $B_3$ inside $\Omega^{(j)}_2$}
\label{sec:B2}

We define a branched lamination, refining in the $b_{32}$ and oscillating in the $b_{13}$ direction and being constant in the $d$ direction.
Since $\Omega_2^{(j)}$ has triangular cross-sections, we need to refine the lamination when approaching the corner.
For this we exploit an argument analogous to that in \cite[Section 5.2]{RT23}.

\begin{lem}\label{lem:2ord-triangle}
Let $B_1$ and $B_3$ be as above.
Then there exist $u^{(j,2)}\in C^{0,1}(\Omega^{(j)}_2;\R^3)$ and $\chi^{(j,2)}\in BV(\rm int(\Omega^{(j)}_2;\{e^{(1)},e^{(3)}\})$ complying with
\begin{equation*}\label{eq:bc-2ord2}
\p_d u^{(j,2)}\equiv 0 \text{ a.e.\ in $\Omega_2^{(j)}$},
\quad
u^{(j,2)}(x)=\uaux(x) \text{ if } x\in\p_0\Omega^{(j)}_2,
\end{equation*}
such that
\begin{equation*}\label{eq:reg-2ord2}
\nabla u^{(j,2)}\in BV(\rm int(\Omega^{(j)}_2;\R^{3\times 3}),
\quad
\|\nabla u^{(j,2)}\|_{L^\infty(\Omega^{(j)}_2)}\lesssim 1,
\end{equation*}
and satisfying
\begin{equation}\label{eq:2ord2-en}
E_\epsilon(u^{(j,2)},\chi^{(j,2)};\Omega^{(j)}_2) \lesssim \Big(\frac{\theta}{2}\Big)^j\frac{r_2^2}{r}+\epsilon \theta^j\frac{r}{r_2}+\frac{r_2 r}{4^j}+\frac{r^3}{(8\theta)^j}.
\end{equation}
\end{lem}
\begin{proof}
As in the previous case, we split $\Omega_2^{(j)}=\Delta \cup T$, where
$$
\Delta := \Big\{x\in\Omega_2^{(j)} : z_1-Y_{j+1}\le-\frac{1}{\sqrt{3}}(z_2 - \frac{L_j}{4})\Big\}
$$
and $T:=\Omega_2^{(j)}\setminus \text{int(}\Delta$ (see Figure \ref{fig:oscB}).

\emph{Step 1: Covering refinement towards the corner inside $\Delta$.}
We cut $\Delta$ into slices orthogonal to $b_{13}$.
The first of such slices is of amplitude $\frac{r_2}{2^j}$, starting from the bottom-right corner of $\Omega_2^{(j)}$, namely we cut along $z_2=-\sqrt{3}(z_1-Y_{j+1})+\frac{L_j}{4}$ since $b_{13}= -\frac{\sqrt{3}}{2}n-\frac{1}{2}b_{32}$. 
We define the quantity $w:=\frac{L_j}{4}\frac{\sqrt{3}}{\sqrt{3}+\frac{L_j}{4H_j}}$ and, for $m\geq 0$, we consider
\begin{multline*}
\omega_m:= \Big\{x\in\Omega : 0\le z_2-\frac{L_j}{4}\le w_m, \\
-\frac{1}{\sqrt{3}}\Big(z_2-\frac{L_j}{4}\Big)-\rho_m\le z_1-\Big(Y_{j+1}-\sum_{m'=0}^{m-1}\rho_{m'}\Big)\le-\frac{1}{\sqrt{3}}\Big(z_2-\frac{L_j}{4}\Big)\Big\}\subset\Delta,
\end{multline*}
where
$$
\rho_m:=\frac{r_2}{2^j}\Big(1-\frac{r_2}{2^jH_j}\Big)^m
\quad \text{and} \quad
w_m:=w\Big(1-\frac{r_2}{2^jH_j}\Big)^{m+1},
$$
see Figure \ref{fig:oscB}.
\begin{figure}
\begin{center}
\includegraphics{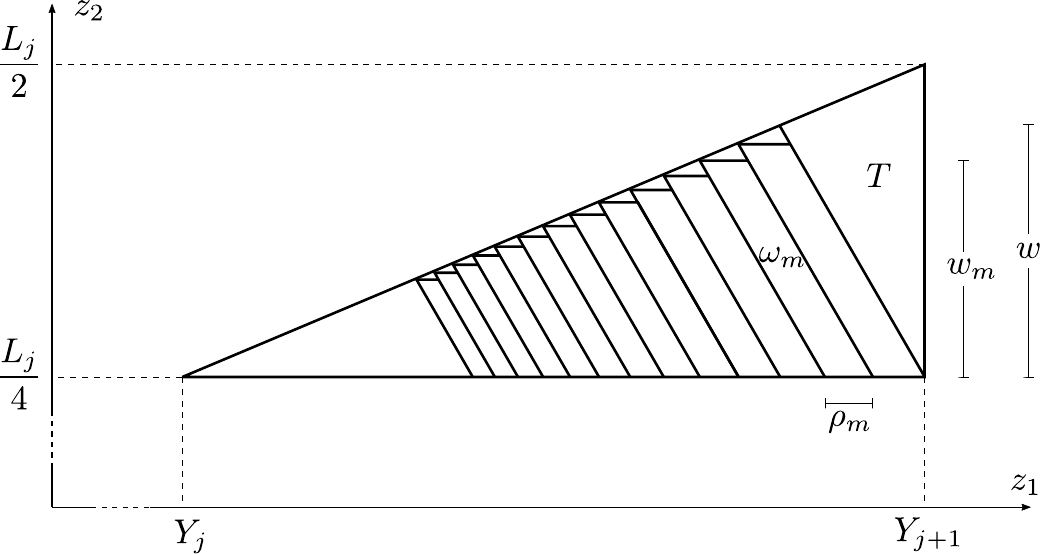}
\caption{The domain subdivision of $\Omega_2^{(j)}$ into sets $\omega_m$ and $T$.}
\label{fig:oscB}
\end{center}
\end{figure}
We stop as $m\le m_0$, with $m_0$ being the largest index such that $\sum_{m=0}^{m_0}\rho_m<H_j-\frac{L_j}{4}$.
Notice that with this choice it holds
\begin{equation}\label{eq:area-refinement}
\Big|\Delta\setminus(\bigcup_{m=0}^{m_0}\omega_m)\Big| \lesssim \frac{L_j}{H_j}\sum_{m\ge0}\rho_m^2+\frac{L_j^3}{H_j} \lesssim \frac{r_2}{2^j}L_j+\frac{L_j^3}{H_j}.
\end{equation}

\emph{Step 2: Filling the corner $T$.} 
This argument is analogous to that of Step 2 of Lemma \ref{lem:2ord1} so we will not repeat it.
We denote as $v_T$ and $\psi_T$ the corresponding constructions.
For these the analogue of \eqref{eq:en-Tl} holds, that is
\begin{equation}\label{eq:en-T}
E_\epsilon(v_T,\psi_T;T) \lesssim \frac{r_2^2}{4^j}+\epsilon\frac{L_j^2 2^{j}}{r_2}+\frac{r_2}{2^j}L_j.
\end{equation}

\emph{Step 3: Construction inside the slices.}
In an identical way as in the proof of Lemma \ref{lem:2ord1}, we define $v^{(m)}\in C^{0,1}(\omega_m;\R^3)$ and $\psi^{(m)}\in BV({\rm int}(\omega_m);\{e^{(1)},e^{(3)}\})$ with $v^{(m)}=\uaux$ on $\p_0 \omega_m$ and such that
\begin{equation}\label{eq:2ord3-est}
E_\epsilon(u^{(m)},\psi^{(m)};\omega_m) \lesssim \frac{\rho_m^3}{w_m}+\epsilon w_m.
\end{equation}
Last be not least, we then define $u^{(j,2)}$ and $\chi^{(j,2)}$ as follows: To this end, we begin by introducing two auxiliary functions
$$
v(x)=
\begin{cases}
v^{(m)}(x) & x\in\omega_m,\, m\in\{0,\dots,m_0\}, \\
\uaux(x) & x\in S\setminus(\bigcup_{m=0}^{m_0}\omega_m),
\end{cases},
\quad
\psi(x)=
\begin{cases}
\psi^{(m)}(x) & x\in\omega_m,\, m\in\{0,\dots,m_0\}, \\
\chiaux(x) & x\in S\setminus(\bigcup_{m=0}^{m_0}\omega_m).
\end{cases}
$$
By \eqref{eq:area-refinement} and \eqref{eq:2ord3-est}, summing over $m$ we get
\begin{equation}\label{eq:en-S2}
\begin{split}
E_\epsilon(v,\psi;\Delta) &\lesssim \sum_{m\ge 0}\Big(\frac{\rho_m^3}{w_m}+\epsilon w_m\Big)+\frac{r_2}{2^j}L_j+\frac{L_j^3}{H_j} \\
& \lesssim \sum_{m\ge 0} \Big(\frac{r_2^3}{8^jL_j}\Big(1-\frac{r_2}{2^jH_j}\Big)^{2m}+\epsilon L_j\Big(1-\frac{r_2}{2^jH_j}\Big)^m\Big)+\frac{r_2}{2^j}L_j+\frac{L_j^3}{H_j} \\
&\lesssim \frac{r_2^3}{4^j r}\frac{(2\theta)^j}{r_2}+\epsilon\frac{r}{2^j}\frac{(2\theta)^j}{r_2} +\frac{r_2 r}{4^j}+\frac{r^3}{(8\theta)^j} \\
& \lesssim \Big(\frac{\theta}{2}\Big)^j\frac{r_2^2}{r}+\epsilon\theta^j\frac{r}{r_2}+\frac{r_2 r}{4^j}+\frac{r^3}{(8\theta)^j}.
\end{split}
\end{equation}
By finally defining
$$
u^{(j,2)}(x)=
\begin{cases}
v(x) & x\in \Delta ,\\
v_T(x) & x\in T,
\end{cases}
\quad \text{and} \quad
\chi^{(j,2)}(x)=
\begin{cases}
\psi(x) & x\in \Delta, \\
\psi_T(x) & x\in T,
\end{cases}
$$
and summing together \eqref{eq:en-T} and \eqref{eq:en-S2}, the claim follows.
\end{proof}

\subsection{Final energy contribution}
\label{sec:final}

We have now everything in place to compute the total energy contribution on the whole domain $\Omega$.

Indeed, we note that the only missing domain is $\Omega_4^{(j)}$ for which we observe that it equals to $\Omega^{(j)}_1\cup\Omega_2^{(j)}$ (up to reflection and shifts). Hence, the lamination inside $\Omega^{(j)}_4$ can be split into laminations analogous to those in $\Omega^{(j)}_1$ and $\Omega^{(j)}_2$, and thus its energy contributions are absorbed by those inside $\Omega^{(j)}_1$ and $\Omega^{(j)}_2$.
We denote as $u^{(j,4)}$ and $\chi^{(j,4)}$ the constructions in there.

\begin{prop}
Let $A_1,A_2,B_1,B_2\in\mathbb{R}^{3\times 3}$ be as above.
Let $r,r_2>0$ such that $r<\frac{1}{2}$ and $r_2<\frac{r}{2}$.
Then there exist $u\in C^{0,1}(\R^3;\R^3)$ and $\chi\in BV_{\rm loc}(\R^3;\{e^{(1)},e^{(2)},e^{(3)}\})$ complying with
\begin{equation*}
\p_d u \equiv 0 \text{ a.e.\ in $\R^3$},
\quad
u(x)=0 \text{ if } x\in\p_0\Omega,
\end{equation*}
such that
\begin{equation*}
\nabla u\in BV_{\rm loc}(\R^3;\R^{3\times 3}),
\quad
\|\nabla u\|_{L^\infty(\R^3)}\lesssim 1,
\end{equation*}
and that
\begin{equation*}
\int_\Omega |e(\nabla u(x))-\chi(x)|^2 dx+\epsilon|D\chi|(\Omega)\lesssim r^2+\Big(\frac{r_2}{r}\Big)^2+r_2+\epsilon\frac{1}{r_2}. 
\end{equation*}
\end{prop}
\begin{proof}
Without loss of generality we can reduce to $r$ and $r_2$ as in the previous sections.
Let $u^{(j,i)}$ be the functions defined in Section \ref{sec:second} and above.
We define 
\begin{multline*}
u(x):=u^{(j,i)}(x-kL_j b_{32}), \quad \mbox{when } x\in \Omega^{(j)}_{i}+k L_j b_{32}, \ j\in\{0,\dots,j_0+1\}, \ i\in\{1,2,3,4\}, \\
k\in\Big\{-\frac{2^{j-1}}{r},\dots,\frac{2^{j-1}}{r}-1\Big\},
\end{multline*}
and
\begin{multline*}
\chi(x):=
\chi^{(j,i)}(x-kL_j b_{32}), \quad \mbox{when } x\in \Omega^{(j)}_{i}+k L_j b_{32}, \ j\in\{0,\dots,j_0+1\}, \ i\in\{1,2,3,4\}, \\
k\in\Big\{-\frac{2^{j-1}}{r},\dots,\frac{2^{j-1}}{r}-1\Big\}.
\end{multline*}
Moreover, we extend the functions $u,\chi$ to $\R^n$, first constantly in $d$ and then to zero outside $[-\frac{1}{2},\frac{1}{2}] n \times [-\frac{1}{2},\frac{1}{2}] b_{32} \times \R d$. 
Gathering \eqref{eq:energy-aux}, \eqref{eq:2ord1-en}, \eqref{eq:en-second} and \eqref{eq:2ord2-en}, since for every generation $j$ we have $\frac{2^j}{r}$ identical copies of $\Omega_i^{(j)}$, by summing over $j$ we obtain
\begin{align*}
E_\epsilon(u,\chi,\Omega) & \lesssim r^2+\epsilon\frac{1}{r}+\sum_{j\ge0}\frac{2^j}{r}\Big(\Big(\frac{\theta}{2}\Big)^j\frac{r_2^2}{r}+\epsilon \theta^j\frac{r}{r_2}+\frac{r^3}{(8\theta)^j}+\frac{r_2 r}{4^j}\Big) \\
& \lesssim \Big(\frac{r_2}{r}\Big)^2+\epsilon\frac{1}{r_2} +r^2+r_2,
\end{align*}
and the result is proved.
\end{proof}

Finally, an optimization argument in the above construction yields the proof of Theorem \ref{thm:upper}.

\begin{proof}[Proof of Theorem \ref{thm:upper}]
An optimization of $r_2$ and $r$ in terms of $\epsilon$ gives $r_2\sim r^2$ and $r\sim \epsilon^\frac{1}{4}$, yielding then
\begin{equation*}\label{eq:2ord-scaling}
E_\epsilon(u,\chi;\Omega) \lesssim \epsilon^\frac{1}{2}.
\end{equation*}
\end{proof}

Towards the setting of the full Dirichlet data, we observe that by a relatively straightforward argument, it is possible to obtain a sub-optimal upper bound construction.

\begin{rmk}
\label{rmk:subopt}
By a simple cut-off procedure towards $\p\Omega\cap\Big\{z_3=\pm\frac{1}{2}\Big\}$, of amplitude $r$ (thus giving energy contribution $r$) it is immediate to show that
$$
E_\epsilon(u_0,\chi_0;\Omega) \lesssim r+\Big(\frac{r_2}{r}\Big)^2+\frac{\epsilon}{r_2}
$$
which optimized gives $r_2\sim r^\frac{3}{2}$, $r\sim \epsilon^\frac{2}{5}$ and hence
$$
\epsilon^\frac{1}{2} \lesssim \inf_{u,\chi}E_\epsilon (u,\chi) \lesssim \epsilon^\frac{2}{5}.
$$

We highlight that, while the scaling $\epsilon^\frac{2}{5}$ is not matching the lower bound (which is expected to be optimal), still it improves the scaling of a simple double laminate, which scales as $\epsilon^\frac{1}{3}$.
\end{rmk}

\subsection{Proof of Theorem \ref{thm:upper2}} 
\label{sec:Thm4}
Last but not least, we give a short sketch of the proof of Theorem \ref{thm:upper2}. This relies on the construction from Section \ref{sec:first} together with the three-dimensional lamination construction which had been introduced in Step 1 in the proof of Proposition 6.3 in \cite{RT21}. 

\begin{proof}[Proof of Theorem \ref{thm:upper2}]
We begin by noting that  $e^{(2)}-e^{(3)} = \frac{3}{2}(e^{(A)}-e^{(B)})$, where $e^{(A)}, e^{(B)}$ are as in Section \ref{sec:grads}. We hence define
\begin{align*}
C^{(2)}:= \frac{3}{2}A + C, \ C^{(2)} = \frac{3}{2}B + C,
\end{align*}
with 
\begin{align*}
C = \begin{pmatrix} 1 & 0 & 0 \\ 0 & - \frac{1}{2} & 0 \\ 0 & 0 & - \frac{1}{2}
\end{pmatrix},
\end{align*}
and observe that 
\begin{align*}
e(C^{(2)}) = e^{(2)}, \ e(C^{(3)}) = e^{(3)}.
\end{align*}
Using the ideas from the construction from \cite[Proposition 6.3]{RT21}, we set
\begin{align*}
u(x_1,x_2,x_3) = \frac{3}{2} \tilde{u}(\max\{|z_1|,|z_3|\},z_2) + C x,
\end{align*}
where $z_j = z_j(x_1,x_2,x_3)$ for $j\in\{1,2,3\}$ are the coordinates from \eqref{eq:coordinates} and $\tilde{u}(z_1, z_2)$ denotes the deformation from Step 1 in Section \ref{sec:first}. By definition of $\tilde{u}$ the function $u$ then satisfies the desired boundary data.
Moreover, as in the first step of \cite[Proposition 6.3]{RT21} we infer that
\begin{align*}
E_{\epsilon}(u,\chi) \lesssim r^2 + \frac{\epsilon}{r},
\end{align*} 
which, by an optimization argument, yields the desired upper bound estimate with full Dirichlet data.
\end{proof}

\section*{Acknowledgements}
Both authors gratefully acknowledge funding by the Deutsche Forschungsgemeinschaft (DFG,
German Research Foundation) through SPP 2256, project ID 441068247, and by the Hausdorff Institute for Mathematics at the University of Bonn which is funded by the Deutsche Forschungsgemeinschaft (DFG, German Research Foundation) under Germany's Excellence Strategy – EXC-2047/1 – 390685813, as part of the Trimester Program on Mathematics for Complex Materials.

\bibliographystyle{alpha}
\bibliography{citations4}

\end{document}